\documentclass[12pt]{article}

\usepackage{amsmath, amsthm, amssymb}  % Math packages
\usepackage{graphicx}                  % For images
\usepackage{geometry}                  % To customize page margins
\usepackage{booktabs}                 % For better tables
\usepackage{cite}                     % For better citation formatting
\usepackage{mathrsfs}                 % For fancy script fonts
\usepackage{float} 
\usepackage{caption}                   % To place figures/tables precisely

\usepackage{algorithm}
\usepackage{algorithmic} 
\usepackage{subcaption} % in your preamble
\usepackage[numbers]{natbib}
\usepackage[colorlinks=true, urlcolor=blue, linkcolor=blue]{hyperref}
\usepackage{booktabs, tabularx}
\usepackage{tcolorbox}
\usepackage{authblk}

\renewcommand{\arraystretch}{1.2} % row height
\usepackage{optidef}   % for optimization problems

\usepackage{pdflscape} % For landscape pages (still needed for sidewaystable) 
\usepackage{rotating}     % For sidewaystable
\usepackage{caption}
\title{\textbf{A Generalized Waist Problem: Optimality Condition and Algorithm}}

\author[1]{Triloki Nath\thanks{Department of Mathematics and Statistics, 
		Deen Dayal Upadhyaya Gorakhpur University, Gorakhpur, India. 
		Email: \texttt{tnverma07@gmail.com}}}

\author[2]{Manohar Choudhary\thanks{Department of Mathematics and Statistics, 
		Dr. Hari Singh Gour Vishwavidyalaya, Sagar, India. 
		Email: \texttt{manoharfbg@gmail.com}}}

\author[2]{Ram K. Pandey\thanks{Department of Mathematics and Statistics, 
		Dr. Hari Singh Gour Vishwavidyalaya, Sagar, India. 
		Email: \texttt{pandeywavelet@gmail.com}}}

\affil[1]{Department of Mathematics and Statistics, Deen Dayal Upadhyaya Gorakhpur University, Gorakhpur, India}

\affil[2]{Department of Mathematics and Statistics, Dr. Hari Singh Gour Vishwavidyalaya, Sagar, India}

\date{}

\theoremstyle{plain}
\newtheorem{theorem}{Theorem}[section]
\newtheorem{lemma}[theorem]{Lemma}
\newtheorem{proposition}[theorem]{Proposition}

\theoremstyle{definition}
\newtheorem{definition}[theorem]{Definition}
\newtheorem{remark}[theorem]{Remark}
\newtheorem{example}[theorem]{Example}

\usepackage{amsopn}

\usepackage{rotating}  % Include this in your preamble
\usepackage{booktabs}  % For nice table rules

\newcommand{\bd}{\operatorname{bd}}
\newcommand{\interior}{\operatorname{int}}

\begin{document}
	
	\maketitle

	\begin{abstract}
		Many years ago John Tyrell a lecturer at King's college London challenged his Ph.D. students with the  following puzzle: show that there is a unique triangle of minimal perimeter with exactly one vertex to lie on one of three  given lines, pairwise disjoint and not all parallel in the space. The problem in literature is known as the waist problem, and only convexity rescued in this case. Motivated by this we generalize it by  replacing lines with a number of  convex sets in the Euclidean space and ask to minimize  the sum of distances connecting the sets by means of closed polygonal curve. This generalized problem significantly broadens its geometric and practical scope in view of modern convex analysis. We establish the existence of solutions and prove its uniqueness under the condition that at least one of the convex sets is strictly convex and all are in general position: each set can be separated by convex hull of others.  A complete set of necessary and sufficient optimality conditions is derived, and their geometric interpretations are explored  to link these conditions with classical principles such as the reflection law of light. To address this problem computationally, we develop a projected subgradient descent method and prove its convergence. Our algorithm is supported by detailed numerical experiments, particularly in cases involving discs and spheres. Additionally, we present a real-world analogy of the problem in the form of inter-island connectivity, illustrating its practical relevance. This work not only advances the theory of geometric optimization but also contributes effective methods and insights applicable to facility location, network design, robotics., computational geometry, and spatial planning.  
	\end{abstract}
	
\textbf{Keywords:} Waist Problem, Geometric Optimization, Convexity, Distance Minimization \\
\textbf{MSC Classification:} 52A20, 51M04

\section{Introduction}
For centuries, mathematicians have been fascinated by the beauty of geometry, inspired by timeless questions of minimizing or maximizing area, perimeter, and volume. One of the earliest problems was solved by Heron of Alexandria, who found the shortest path joining a point on a straight line to two given points in the same plane (the non-trivial case is when both points are on the same side of the plane). This problem later became known as the Heron problem. In the 17th century, this fascination continued with Heron's first generalizations as the Fermat-Torricelli Point problem \cite{AnMuSt07}, which seeks a point in the plane that minimizes the sum of distances from three given points in the same plane.\\
\indent Pushing the limits of mathematical insight, Endre Weiszfeld\cite{We37, WePl09} asks for finding a point that minimizes distances to multiple fixed points. Later, Fagnano\cite{AnMuSt07} introduced a geometric problem, famously known as the Fagnano problem, adding new depth to the study of optimization. Inspired by Fagnano's classical problem, recently in \cite{NaCh25}, we have generalized it for the first time to non-convex quadrilaterals. A similar motivation leads us to consider the waist problem, classically defined as identifying the triangle of minimal perimeter whose vertices each lie on one of three given lines in the space, as described in~\ref{Some_Optimization_Problems}.  In modern times, these classic problems and its generalizations have applicability in many practical situations, such as location science (facility planning), network design, and robotic navigation, showing the need to reinvestigate these generalizations. Recognizing the limitations of linear structures in practical purposes, we generalize the waist problem by replacing lines with convex sets. For illustration, a concrete example of the generalized waist problem is elaborated in \ref{sec:island_analogy}. Proceeding with the existence and uniqueness of a solution, we develop an optimal condition and demonstrate its numerical implementations using various examples.\\
\indent We structure this paper as follows: \ref{Some_Optimization_Problems} revisits the foundational geometric optimization problems, such as the classical Heron problem, Fermat-Torricelli Point problem, Fagnano problem, and waist problem. In~\ref{Formulation}, we present our main focus—the formal generalization of the waist problem involving multiple convex sets.
In \ref{sec:island_analogy}, we show a real-world analogy: inter-island connectivity to our generalized problem, exhibiting that this generalization is more fruitful. In~\ref{Optimality_Conditions}, we demonstrate that even without assuming the boundedness of the convex sets, the \textit{ generalized waist problem} admits a solution, and the solution is unique if at least one set is strictly convex under general positioning of the sets. We also rigorously develop a necessary and sufficient optimality conditions for this problem and its geometrical interpretation. The geometrical interpretation sheds light on how these problems are connected with the fundamental laws of reflection of light. 
In~ \ref{Convergence}, we provide a computational method, namely the projected subgradient descent algorithm. We have shown that our method is convergent. \ref{Numericalexample} presents illustrative numerical examples that highlight the effectiveness and efficiency of our proposed method, notably showing massive improvement in convergence compared to existing approaches. In particular, it is illustrated for two similar cases, namely for  three discs and three spheres. Finally, \ref{conclusion} reflects the strengths and limitations of our approach and outlines future research directions.

\subsection{Some Classical Optimization Problems} \label{Some_Optimization_Problems}
In this section we begin with classical Heron's problem and proceed with its significant generalizations. we observe that how these generalizations occur naturally and found to be applicable in diverse areas. We recall these classical problems to understand the motivation behind our generalization.
\begin{enumerate}
	\item \textbf{Heron's Problem\cite[Problem 1.1.1]{AnMuSt07}:} One of the oldest optimization problems is given by Heron of Alexandria in the first century AD, known as Heron's problem, it asks: given two points on the same side of a straight line in the same plane, find the point on the line that minimizes the sum of distances to the given two points. Some significant generalizations of Heron's problem are discussed in \cite{MoNaSa12, MoNaSa11, ChLa14}.
	\begin{figure}[H] 
		\centering
		\includegraphics[scale=.15]{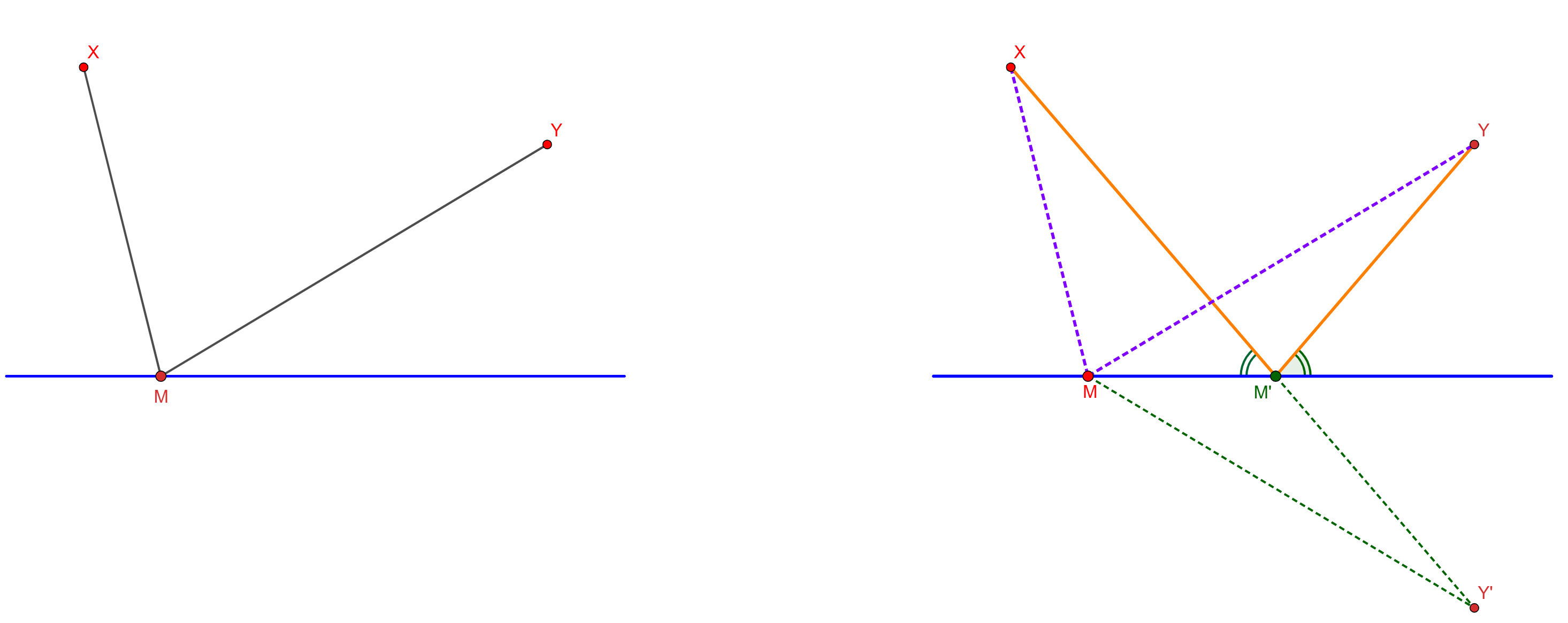} 
		\caption{Heron's problem and its Solution}
		\label{fig:HeronProblem}
	\end{figure}
	\item \textbf{Fermat-Torricelli Problem \cite[Problem 2.4.4]{BrTi11}:}
	The Fermat–Torricelli problem was initially posed by Pierre de Fermat in the 17th century as a first generalization of the Heron problem, replacing the line with the whole plane and taking three points instead of two points. Later, it is solved by Evangelista Torricelli, finding a point in the plane that minimizes the sum of its distances to three given points. Geometrically, if all interior angles of the triangle formed by these points are less than $120^{\circ}$, the optimal point known as the Fermat point lies inside the triangle and forms $120^{\circ}$ angles with each pair of connecting lines to the triangle's vertices. However, if any angle of the triangle is $120^{\circ}$ or greater, the vertex at that angle itself becomes the solution to the problem. For more than three points, Endre Weiszfeld \cite{We37, WePl09} introduced an iterative procedure, now known as the Weiszfeld algorithm, to approximate the point, minimizing the total distance. Since then, numerous generalizations and extensions of the Fermat–Torricelli problem have been explored; see \cite{Br95, MoNaSa11, NaHoAn14} for further details.
	\begin{figure}[H] 
		\centering
		\includegraphics[scale=.15]{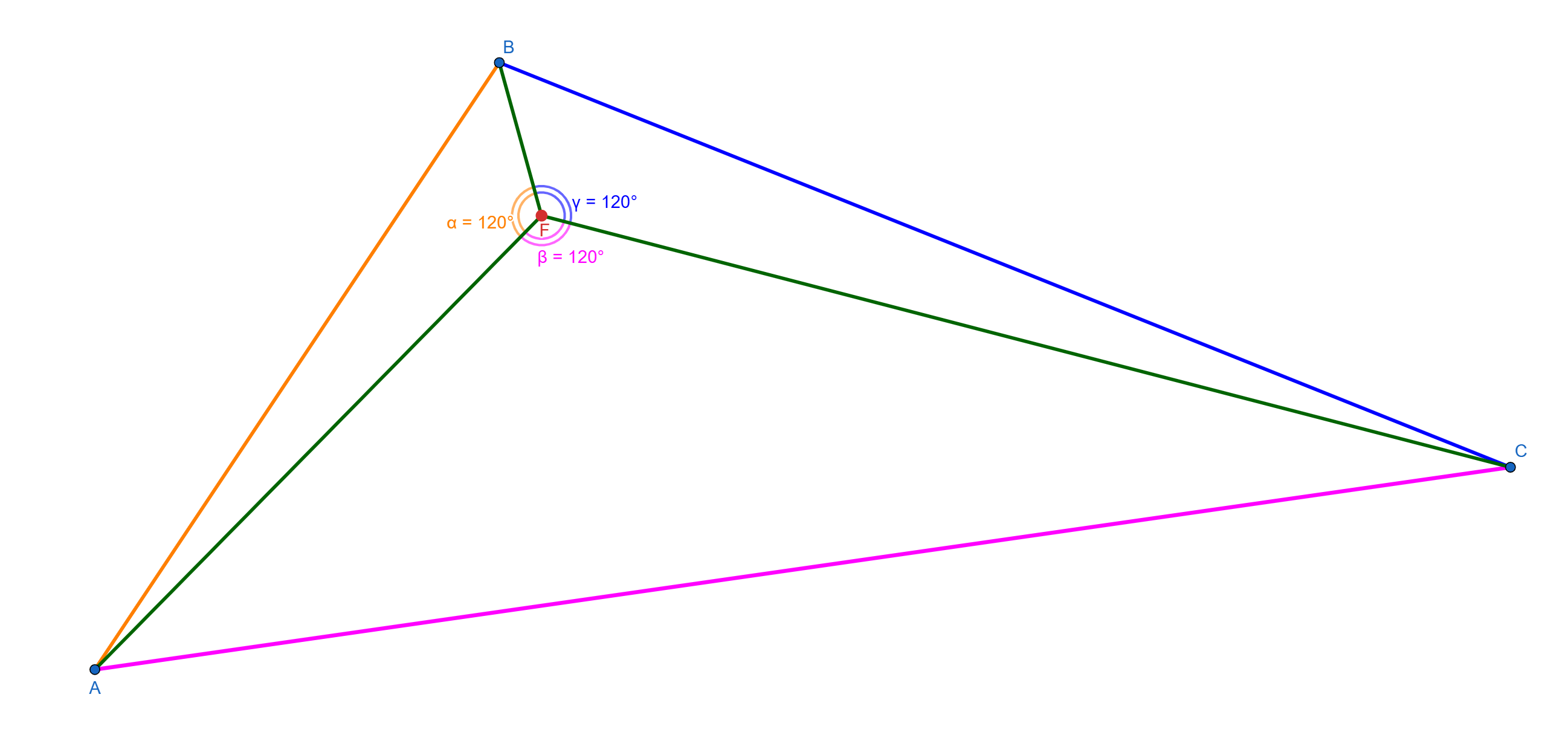} 
		\caption{Fermat-Torricelli problem and its Solution}
		\label{fig:FermatProblem}
	\end{figure}
	
	The following two classical geometrical problems are main ingredients to our generalization. Surprisingly, these two problems remained largly unnoticed untill now through the lens of modern convex analysis. 
	\vspace{0.25cm}
	\item \textbf{Fagnano's Problem\cite[Problem 1.1.3]{AnMuSt07}:} Building on the above foundational problems (Heron and Fermat-Torricelli), the Fagnano problem, proposed by Giovanni Fagnano in 1775, is to minimize the perimeter of a triangle inscribed within a given acute-angled triangle. Fagnano's solution shows that the minimal perimeter is achieved when the inscribed triangle is the orthic triangle, formed by connecting the feet of the altitudes of the original triangle. This result not only illustrates the relationship between geometric constructs but also highlights the importance of reflective properties and symmetries in optimization.
	\begin{figure}[H] 
		\centering
		\includegraphics[scale=.15]{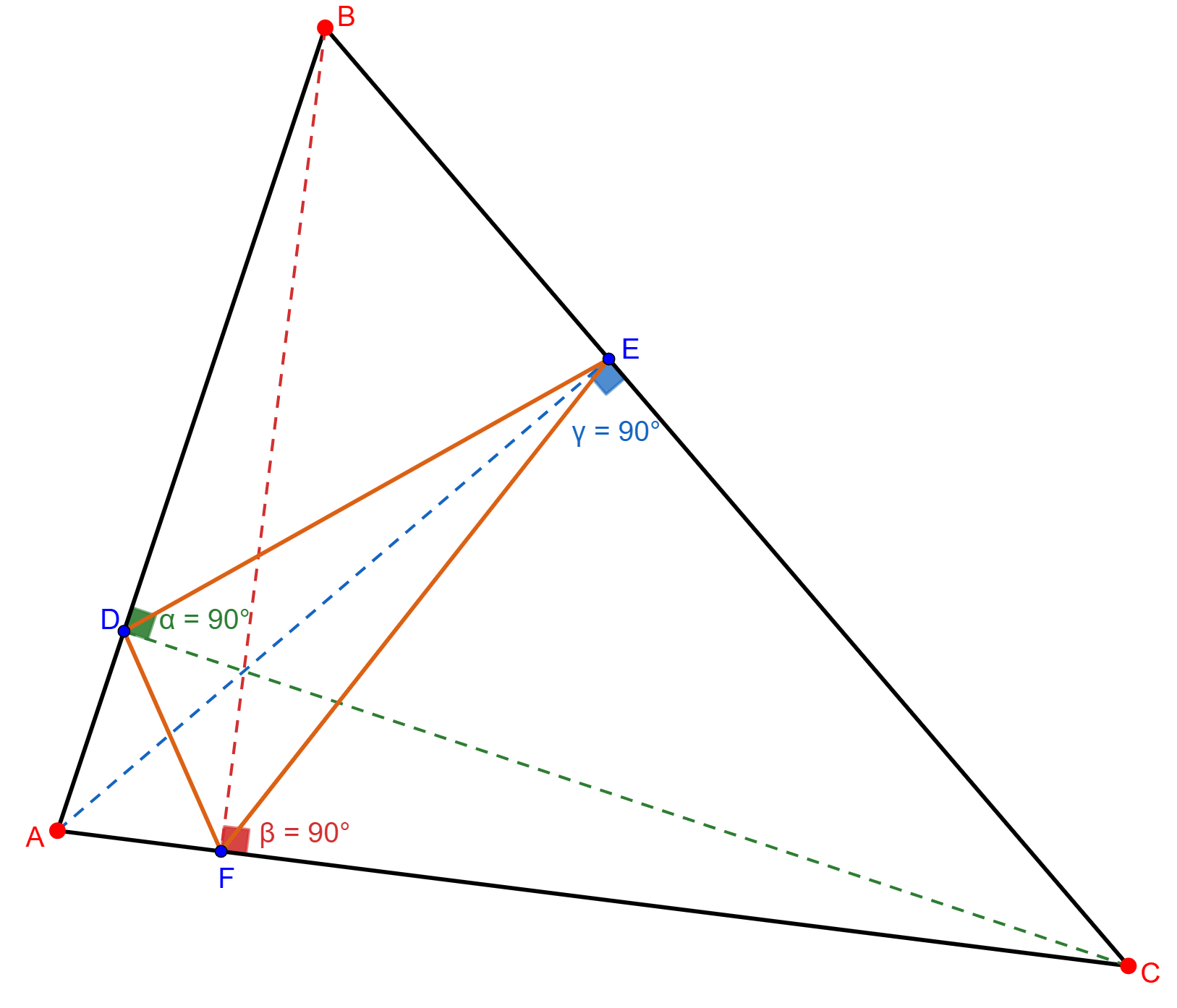} 
		\caption{Fagnano problem and its solution(Orthic Triangle)}
		\label{fig:FagnanoProblem}
	\end{figure}
	
	\item \textbf{Waist Problem\cite[Problem 4.3.4]{BrTi11}:} The waist problem, initially proposed by John Tyrrell, poses a fascinating geometric and optimization challenge: find the minimal \textit{``waist''} formed by an elastic band stretched around three fixed lines in three-dimensional space. These lines are assumed to be pairwise disjoint and not all parallel.

	\indent Geometrically, imagine wrapping of an elastic band around three rigid wires in the space. Due to its elasticity, the band naturally contracts to a position where the total length of the triangle it forms, connecting one point on each line, is minimized. This final configuration is referred to as the \textbf{waist}. Formally, this waist problem is modeled as the following optimization problem:
	
	\begin{equation}
		\min f(a_1, a_2, a_3) = \lvert a_1 - a_2 \rvert + \lvert a_2 - a_3 \rvert + \lvert a_3 - a_1 \rvert
	\end{equation}
	
	subject to \( a_i \in l_i \), $i=1,2,3$ where \( l_1, l_2, l_3 \) are the given lines in the space.\\
	\indent Initially, standard approaches, such as geometric reasoning and applying Fermat’s theorem to find stationary points, failed to establish the uniqueness of the solution. However, the breakthrough came through the lens of convexity. The objective function \( f \) is shown to be strictly convex. Consequently, it possesses a unique global minimum \cite{BrTi11}. This elegant insight bypasses the need for heavy computations and secures the uniqueness of the solution.\\
	\indent From a geometric perspective, reaching the minimal total length, the elastic band forms a triangle where each perpendicular to the line at each vertex is an angle bisector. In other words, each pair of two sides of the triangle is equally inclined with the line having the common vertex—just as in Heron’s shortest path problem, where the path of minimal distance reflects with equal angles. This equal-angle condition ensures the points reach a natural geometric equilibrium. We observe that the solution criterion to Heron's shortest path problem is the key to solve the waist problem.
	
	\begin{figure}[H]
		\centering
		\includegraphics[scale=0.15]{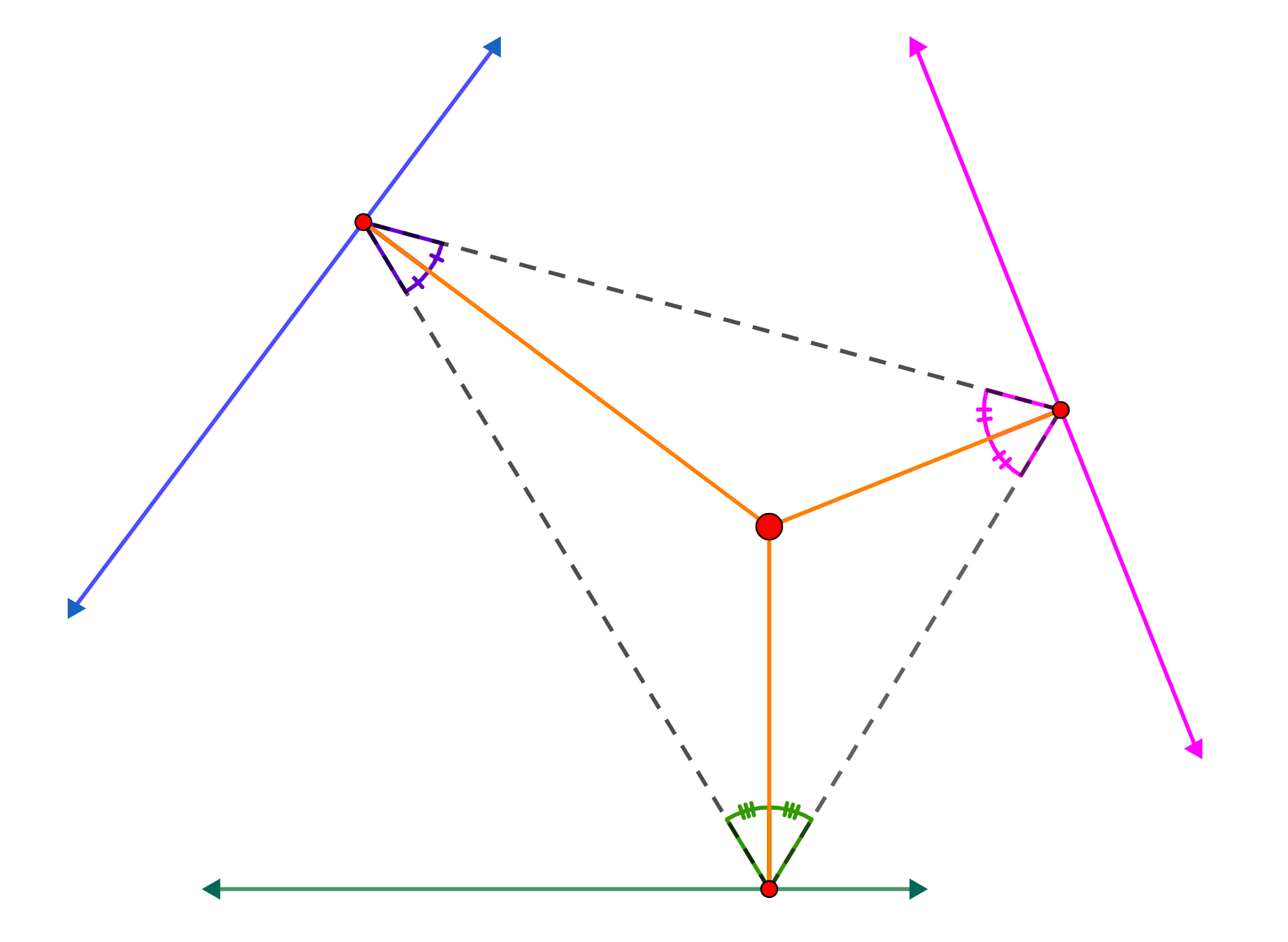}
		\caption{Illustration of the unique waist formed by three lines in space}
	\end{figure}
\end{enumerate}	

Relying upon this idea, let us increase just one more line in the classical waist problem, formally suppose there are four lines in the space  $\mathbb{R}^3$, and all are pairwise disjoint, and all are not parallel. Again, one may ask if this generalized waist problem has a unique waist.
This motivates us, instead of this problem, to consider the general case where the lines are replaced by a number of closed and convex sets. We now introduce a generalized waist problem, which is the main content of our paper, formulated in the following \ref{Formulation}.

\subsection{Generalized Waist Problem} \label{Formulation}
We consider the $n$-dimensional Euclidean space $\mathbb{R}^n,$ with standard inner product  $\langle \cdot, \cdot \rangle,$ and induced Euclidean norm $ \| \cdot \|$.\\
\indent Recall that a set $ C \subset \mathbb{R}^n $ is called \emph{convex} if, for any two points $ x, y \in C $ and any $ \lambda \in [0,1] $, the point
$\lambda x + (1 - \lambda) y \in C.$
That is, the line segment connecting any two points in $ C $ lies entirely within the set. A set $ C $ is said to be \emph{strictly convex} if, for any two distinct points $ x, y \in C $ and any $ \lambda \in (0,1) $, the point
$\lambda x + (1 - \lambda) y \in \operatorname{int}(C),$
where $ \operatorname{int}(C) $ denotes the interior of $ C $. In other words, a strictly convex set does not contain any line segments on its boundary, equivalently,  the interior of the segment between any two distinct points lies strictly inside the set.\\
\indent Throughout the paper, we let \(C_1, C_2, \dots, C_m \subset \mathbb{R}^n\) be nonempty, closed, and convex sets. Let $C := C_1 \times C_2 \times \cdots \times C_m \subset$ $\underbrace{\mathbb{R}^{n} \times \mathbb{R}^{n} \times \dots \times \mathbb{R}^{n}}_{\text{m times}}$.  For any point \(a = (a_1, a_2, \dots, a_m) \in C\), with each \(a_i \in C_i\), we define the \emph{cyclic sum of Euclidean distances} as  
\begin{equation}\label{Distance_Function}
	D(a) := \sum_{i=1}^{m} \|a_i - a_{i+1}\|, \quad \text{where}~ ~ a_{m+1} = a_1.
\end{equation}
Then our problem is 
\begin{equation}
	\min_{a \in C} D(a)
	\label{objective}
\end{equation}
\begin{figure}[H] 
	\centering	
	\includegraphics[scale=.12]{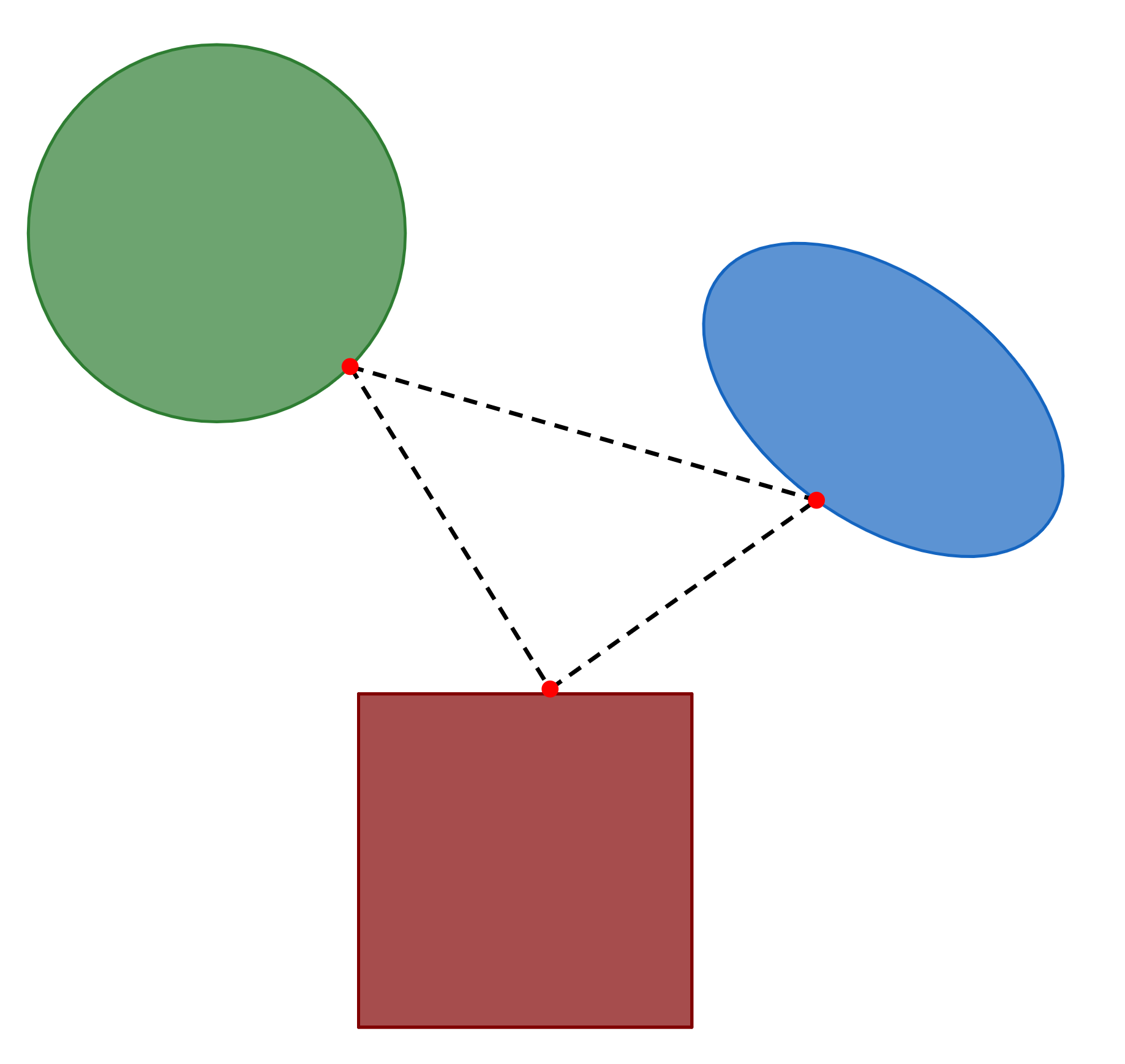} 
	\caption{Illustration of the generalized waist problem for three convex sets.}
\end{figure}

Geometrically, the problem is to determine a cyclic $m$-tuple of points $(a_1,\allowbreak a_2,\allowbreak \dots,\allowbreak a_m)$
with each $a_i$ constrained to lie in $C_i$, that minimizes the total perimeter of the closed polygonal chain formed by joining the points in sequence from $a_1$ to $a_2$, then $a_2$ to $a_3$, and so on, closing the loop from $a_m$ back to $a_1$. By \emph{closed polygonal chain} (see~\ref{fig:Convex_Polytope}), we mean a closed curve in the space formed by joining line segments end to end.
The optimization problem is defined for a specific sequence (or ordering) of the convex sets, say $(C_1, C_2, \ldots, C_m)$. Permuting this sequence generally results in a distinct optimization problem, potentially producing a different optimal value and/or a different optimal solution set. 

Since $C = C_1 \times C_2 \times \dots \times C_m$ is the finite product of a convex set $C_i,$  consequently $C$ is convex. Since norm is a convex function. It follows that $D(a)$, being the sum of convex functions, is convex. Therefore, problem~\ref{objective} is a convex optimization problem. This gives us a first glimpse that the elegant techniques from convex analysis and convex optimization must be applicable and we show that these tools are crucial to solve the proposed problem in an efficient manner.

\begin{figure}[H] 
	\centering
	
	\includegraphics[scale=.42]{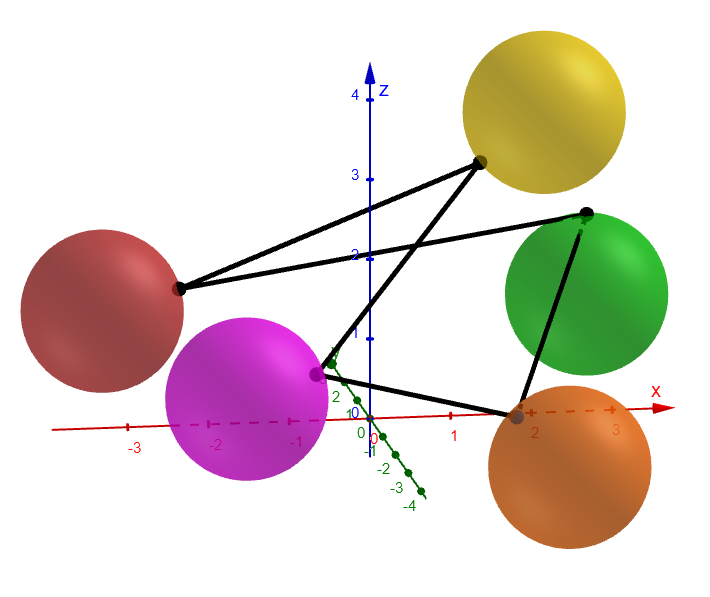} 
	\caption{Illustration of the closed polygonal chain in $\mathbb{R}^3$}
	\label{fig:Convex_Polytope}
\end{figure}
To the best of the author's knowledge, no generalization of the waist problem has been considered in the existing literature. This work presents such a generalization for the first time. Thus, our problem is new, and the following real word analogy given in \ref{sec:island_analogy} shows that it is a significant generalization of the classical waist problem, which directly resembles to distance minimization problems.
\subsection{Real-World Analogy: Inter-Island Connectivity in the Andaman Region}
\label{sec:island_analogy}
To motivate our generalized waist problem, we present a real-world geographic analogy inspired by the Andaman and Nicobar Islands, a remote Indian archipelago in the Bay of Bengal. Consider four significant landmasses in this region: \textit{Little Andaman}, \textit{North Sentinel Island}, \textit{South Andaman Island}, and \textit{Rutland Island}, each represented as a closed, convex subset of $\mathbb{R}^2$. Within each island, a specific \textit{convex subregion} is identified, illustrated by red boundaries in Figure~\ref{fig:IslandIllustration}, corresponding to feasible zones for infrastructure placement (e.g., docking points, communication hubs, or relief centers).

We aim to establish a \textit{closed-loop connection} (shown as black dashed lines) joining one point in each convex region to form a \textit{closed polygonal chain} of minimal total length. This problem is structurally identical to the one defined in equation~\eqref{objective}, where each convex set $C_i$ corresponds to a designated region on an island, and each selected point $a_i \in C_i$ represents a chosen node for inter-island connection.

The \textit{cyclic minimization} of inter-island Euclidean distances models efficient planning of physical or virtual infrastructure networks, such as:
\begin{itemize}
	\item Optimal placement of maritime transport docks or bridges,
	\item Minimization of underwater communication cable lengths,
	\item Emergency service routing between constrained zones on remote islands.
\end{itemize}

\begin{figure}[H]
	\centering
	\includegraphics[width=0.42\textwidth]{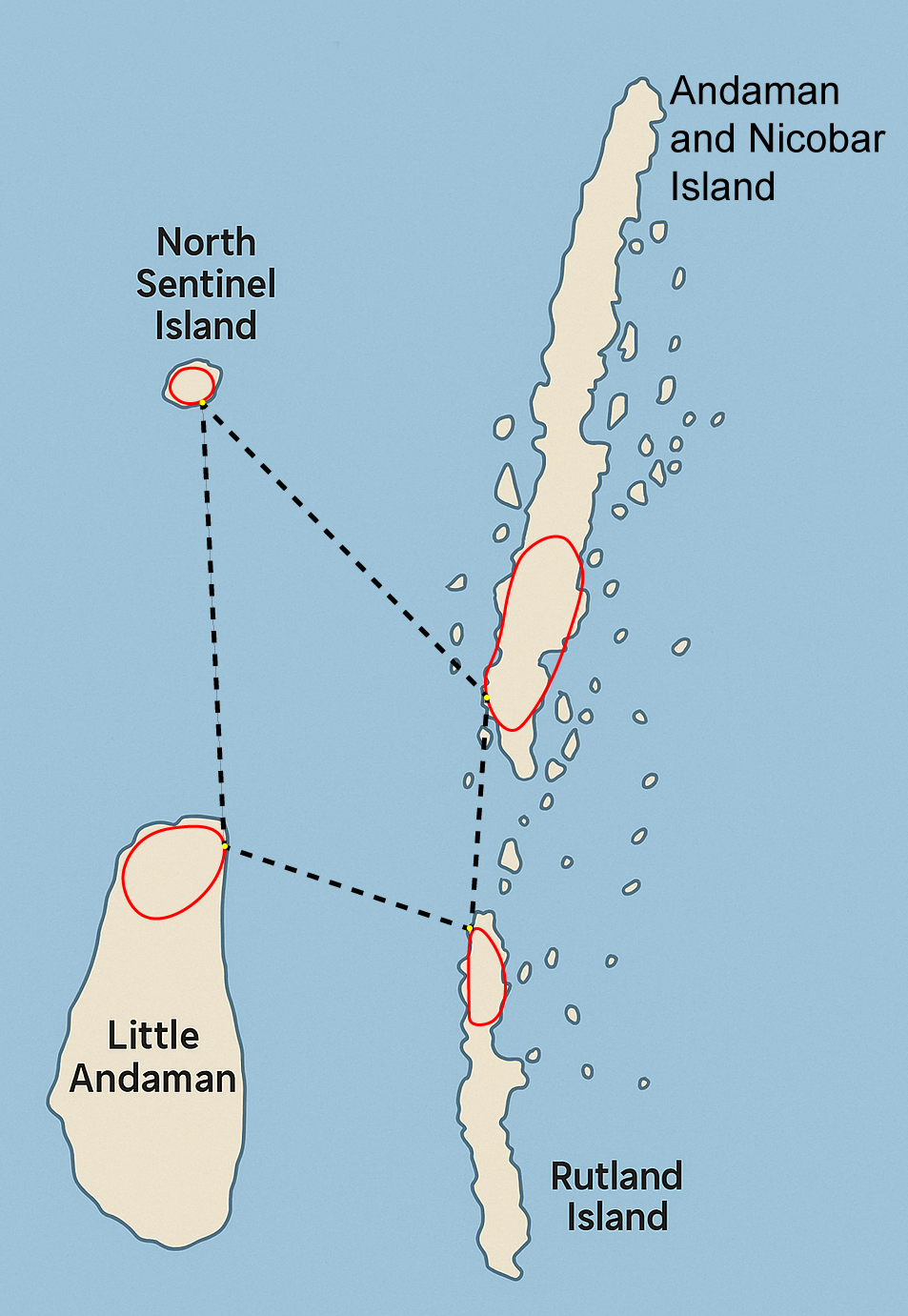}
	\caption{Real-world analogy of the waist problem using four islands in the Andaman region.}
	\label{fig:IslandIllustration}
\end{figure}

This visualization underscores the practical relevance of the proposed optimization problem in scenarios where spatial constraints, convexity, and network efficiency naturally arise. The generalized waist formulation thus offers a principled framework for addressing such constrained multi-location optimization tasks.

\section{Preliminaries}
In this section, we recall some basic definitions and results of convex analysis for a detailed study of the proposed problem in this paper. We refer to \cite{BeNeOz03, HiLe96, Be14, Ro70, DhDu11} for a comprehensive study of these materials and additional backgrounds.

The notion of convex function is a foundational concept in our analysis. The geometry of convex sets and convex functions will allow us to employ powerful tools such as \textit{normal cones}, \textit{subdifferentials}, and \textit{projection mappings.}

Before defining convex functions that may take on infinite values (such as indicator functions), it is important to identify the region where these functions are finite. This domain determines the relevant domain over which optimization and analysis can proceed.

\begin{definition}The \textit{effective domain} of a function $f : \mathbb{R}^n \to (-\infty, +\infty]$ is the set
	$$
	\operatorname{dom} f = \{ x \in \mathbb{R}^n \mid f(x) < +\infty \}.
	$$
	It consists of all points where the function $f$ attains finite values.
\end{definition}

\begin{definition} The \textit{epigraph} of a function $f : \mathbb{R}^n \to \mathbb{R}$ is the set
	$$
	\operatorname{epi} f = \{ (x, t) \in \mathbb{R}^n \times \mathbb{R} \mid f(x) \leq t \}.
	$$
	This set represents all points lying on or above the graph of $f$ in $\mathbb{R}^{n+1}$. 
\end{definition}

\begin{definition} A function $f : \mathbb{R}^n \to (-\infty, +\infty]$ is called convex if its effective domain $\operatorname{dom} f$ is convex, and for all $x, y \in \operatorname{dom} f$ and $\lambda \in (0,1)$,
	$$
	f(\lambda x + (1 - \lambda) y) \leq \lambda f(x) + (1 - \lambda) f(y).
	$$
	and $f$ is said to be strictly convex if above inequality is strict for all $ x \ne y \in \operatorname{dom} f$ and $\lambda \in [0,1]$
	Equivalently, $f$ is (strictly)  convex if and only if its epigraph $\operatorname{epi} f$ is a  (strictly)  convex set.
\end{definition}

To establish the uniqueness of the solution our problem \eqref{objective}, we draw upon two key ideas from convex analysis: the convex hull and the separation theorem. The convex hull provides a way to enclose all feasible points within the smallest possible convex set, allowing us to examine the structure of potential solutions. Meanwhile, the separation theorem ensures that under appropriate conditions, distinct feasible points can be separated using hyperplanes. Together with strict convexity of just one set, these concepts are essential in proving that the solution to our problem is not only feasible but also unique. This naturally leads us to recall the following definition.

\begin{definition} \label{def:convex_hull}	Given a set \( S \subseteq \mathbb{R}^n \), its \emph{convex hull}, denoted \(\mathrm{conv}(S)\), is the smallest convex set containing \(S\). Equivalently, \cite[Definition 6.11 ]{Be14}
	\[
	\mathrm{conv}(S)
	\;=\;
	\Bigl\{
	\sum_{i=1}^{m} \lambda_i x_i
	\;\Big\vert\;
	x_i \in S, \;
	\lambda_i \ge 0, \;
	\sum_{i=1}^{m} \lambda_i = 1, \;
	m \in \mathbb{N}
	\Bigr\}.
	\]
	In other words, any point in \(\mathrm{conv}(S)\) can be written as a convex combination of points in \(S\).
\end{definition}

The Separation of two convex sets by means of a hyperplane plays an important role in optimization and analysis. In particular, the alternative theorem arising in study of optimality conditions are consequences of separation theorem \cite[p.124]{HiLe96}. Here are these important notions.    

\begin{definition} Let \( C_1 \) and \( C_2 \) be two nonempty, disjoint convex sets in \( \mathbb{R}^n \). We say that \( C_1 \) and \( C_2 \) are \textbf{separable} if there exists a nonzero vector \( a \in \mathbb{R}^n \) and a scalar \( \alpha \in \mathbb{R} \) such that
	\[
	\langle a, x \rangle \leq \alpha \quad \forall x \in C_1 \quad \text{and} \quad \langle a, y \rangle \geq \alpha \quad \forall y \in C_2.
	\]
	This means a hyperplane \( H = \{ z \in \mathbb{R}^n : \langle a, z \rangle = \alpha \} \) separates the two sets, with one on each side or possibly touching the hyperplane.
	
	\vspace{0.3cm}
	
	If, in addition, the inequality is strict, i.e., there exists \( a \in \mathbb{R}^n \setminus \{0\} \) and \( \alpha \in \mathbb{R} \) such that
	\[
	\langle a, x \rangle < \alpha \quad \forall x \in C_1 \quad \text{and} \quad \langle a, y \rangle > \alpha \quad \forall y \in C_2,
	\]
	then \( C_1 \) and \( C_2 \) are said to be \textbf{strictly separable}.
	
\end{definition}

\begin{remark}
	Even when two convex sets are closed and disjoint, strict separation is not guaranteed unless additional conditions (e.g., compactness of at least one set) are satisfied.
\end{remark} 

Let
$A = \{(x,y) \in \mathbb{R}^2 \mid y \ge 0\},$ and $B = \{(x,y) \in \mathbb{R}^2 \mid xy \le -1,\; x > 0\}.$
The sets $A$ and $B$ are disjoint, closed, and convex; however, no hyperplane can strictly separate them (see~\ref{fig:separation}).

\begin{figure}[H]
	\centering
	\includegraphics[width=0.4\textwidth]{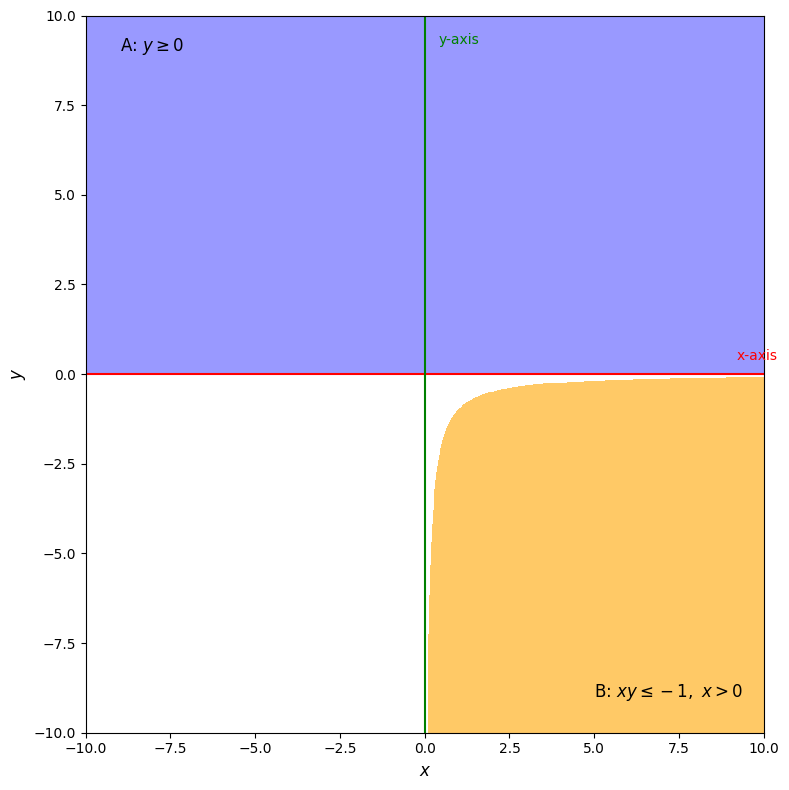}
	\caption{Convex sets \(A\) and \(B\) are closed and disjoint, but cannot be strictly separated by a hyperplane.}
	\label{fig:separation}
\end{figure}
However, the following theorem guarantees strict separation when one of the sets is compact.  	
\begin{theorem}[\textit{Strict Separation Theorem}]
	Let \( C_1 \) and \( C_2 \) be two nonempty, disjoint, closed convex subsets of \( \mathbb{R}^n \), and suppose that at least one of them is compact. Then there exist a nonzero vector \( a \in \mathbb{R}^n \) and a scalar \( \alpha \in \mathbb{R} \) such that
	\[
	\langle a, x \rangle < \alpha \quad \forall x \in C_1, \qquad \langle a, y \rangle > \alpha \quad \forall y \in C_2.
	\]
\end{theorem}

A fundamental question for any problem is whether a solution actually exists. For minimization problems over potentially unbounded domains for e.g.  whole \( \mathbb{R}^n \), the objective function might decrease indefinitely. The concept of coercivity is vital to prevent this. A coercive function grows infinitely large as its argument moves infinitely far from the origin, ensuring that the minimum must occur within some bounded region.

\begin{definition}\cite[Definition 2.31]{Be14}
	A function \( f: \mathbb{R}^n \to \mathbb{R} \) is said to be \textit{coercive} if 
	\[
	\| x \| \to \infty \implies f(x) \to \infty,
	\]
	
	In other words, as the magnitude of \( x \) increases indefinitely, the value of \( f(x) \) also increases without bound.
\end{definition}

\begin{theorem}\label{coercivity_theorem}
	\cite[Corollary 1.7]{BrTi11}
	If a function \( f: \mathbb{R}^n \to \mathbb{R} \) is continuous and coercive, then the minimization problem 
	\[
	\min_{x \in \mathbb{R}^n} f(x)
	\]
	has a global minimum.
	
\end{theorem}	

A key aspect of our research involves evaluating how near a point lies to other points or to its respective constraint set \( C_i \). The \emph{distance function} provides a precise way to express the minimal distance between a point and a set.

\begin{definition} Let \( C \subset \mathbb{R}^n \) be a nonempty set. The \textit{distance function} \( d_C : \mathbb{R}^n \to [0, +\infty) \) of the set \( C \) is defined as:
	\[
	d_C(a) = \inf_{b \in C} \| a - b \|.
	\] 
	
	It may be noted that $	d_C(\cdot)$	is a convex function if and only if $C$ is a convex set.
	
\end{definition}	

Achieving this minimum distance involves finding a specific point within the set $C$. The \emph{Euclidean projection} identifies this closest point (or points). For the closed convex sets $C_i$, this projection is unique. The projection operator is not just theoretical it often appears directly in algorithms (e.g., projected subgradient descent methods) used to solve constrained optimization problems, ensuring iterates remain feasible.

\begin{definition}	Let $C \subseteq \mathbb{R}^n$ be a nonempty set.
	For each $a \in \mathbb{R}^n$, the \emph{Euclidean projection} of $a$ onto $C$ is defined by
	\[
	\operatorname{proj}_C(a)
	\;:=\;
	\{\,b \in C \;\mid\; \|a - b\| \;=\; d_C(a)\},
	\quad\text{where}\quad
	d_C(a)
	\;:=\;
	\inf_{b \in C}
	\|a - b\|.
	\]
	Furthermore, if $C$ is nonempty, closed, and convex, then for each $a \in \mathbb{R}^n$ $\operatorname{proj}_C(a)$ is unique. Thus the set $\operatorname{proj}_C(a)$ is a singleton, and 
	\[
	d_C(a) = \| a - \operatorname{proj}_C(a) \|,
	\]
	
\end{definition}

In our problem~\ref{objective}, it is essential that each point $ a_i $ remains within its respective convex set $ C_i $.

A powerful technique from convex analysis to handle such constraints is the \emph{indicator function}. This non-negative extended real value function elegantly transforms the constrained problem into an equivalent unconstrained one by assigning an infinite penalty to points outside the feasible set.

\begin{definition} 	For a set $K \subset \mathbb{R}^n$, the \textit{indicator function} $\delta_K : \mathbb{R}^n \to \{0, +\infty\}$ is defined as
	$$
	\delta_K(a) = \begin{cases}
		0, & \text{if } a \in K, \\
		+\infty, & \text{if } a \notin K.
	\end{cases}
	$$
	
\end{definition}

This function effectively penalizes any point outside $K$ by assigning it an infinite cost, thus ensuring that the optimization problem considers only points within $K$. Addition of this typical extended real valued function helps us to convert constrained optimization problem~\ref{objective} into an equivalent unconstrained optimization problem.

\begin{equation} 
	\min_{(a=a_1, a_2, \dots, a_m) \in\, \mathbb{R}^{mn}} \Biggl\{
	\sum_{i=1}^{m} \|a_i - a_{i+1}\| +  \delta_{C}(a)
	\Biggr\},
	\label{unconstrainedobjective}
\end{equation}
where $\mathbb{R}^{mn}$ stands for $\underbrace{\mathbb{R}^{n} \times \mathbb{R}^{n} \times \dots \times \mathbb{R}^{n}}_{\text{m times}}$.

This reformulation helps us to apply the tools from convex analysis and nonsmooth optimization methods. Using this reformulation, it is much easier to implement line search algorithm to solve the problem computationally.

To derive optimality conditions and to develop a computational method for solving the problem, the notions of normal cone and subgradient are ubiquitous.   

A key step in this process is applying the concept of subgradients to obtain explicit optimality conditions for our objective function in~\ref{unconstrainedobjective}. This requires to compute subdifferentials—especially for distance functions, as well as for compositions and sums of functions. In particular, since our objective function involves terms for measuring distances—either from points to constraint sets \( C_i \), or between points, Note that $d_{\{a_j\}}(a_i)=\|a_i-a_j\|$ is distance of $a_i$ from the singleton set $\{a_j\}$. Therefore it is essential to compute the subdifferential of the distance function \( d_C(a) \). The resulting formula depends critically on whether the point \( a \) lies inside or outside the set \( C \).

The following notion is a nonsmooth generalization of a normal to a surface and is referred to as \textit{normal cone}, which plays central role even in case of smooth optimization.

\begin{definition}
	Let $C \subset \mathbb{R}^n$ be a convex set, and let $x \in C$. The \textit{normal cone} to $C$ at $x$, denoted by $N_C(x)$, is defined as
	$$
	N_C(x) = \{ v \in \mathbb{R}^n \mid \langle v, y - x \rangle \leq 0, \ \forall y \in C \},
	$$
	where $\langle \cdot, \cdot \rangle$ denotes the standard inner product in $\mathbb{R}^n$. Geometrically, the normal cone consists of all vectors $v$ that form an obtuse angle with any vector pointing from $x$ to another point $y$ in $C$. 
\end{definition}

The following result provides a nice tool to compute normal cone to cartessian product of convex sets, and it is necessary what we follows.

\begin{proposition}\cite[Page 59, Proposition 2.39]{DhDu11} 
	Consider two closed convex sets $C_i \subset \mathbb{R}^{n_i}$, for $i = 1, 2$. 
	Let $a_i \in C_i$, for $i = 1, 2$. Then
	\[
	N_{C_1 \times C_2}\bigl((a_1, a_2)\bigr)
	= 
	N_{C_1}(a_1) 
	\times 
	N_{C_2}(a_2).
	\]
	\label{Normalcone_Product}
\end{proposition}

\begin{definition} Let $f : \mathbb{R}^n \to (-\infty, +\infty]$ be a convex function. A vector $v \in \mathbb{R}^n$ is called a \textit{subgradient} of $f$ at $x \in \operatorname{dom} f$ if
	$$
	f(y) \geq f(x) + \langle v, y - x \rangle, \quad \forall y \in \mathbb{R}^n.
	$$
	The set of all subgradients at $x$ is called the subdifferential of $f$ at $x$, and it is denoted $\partial f(x)$.
\end{definition}

For the indicator function \( \delta_C \) of a convex set \( C \), the subdifferential at a point \( x \in C \) is equal to the normal cone to \( C \) at \( x \) ~\cite[Example 3.7, p 85]{MoNa23}. That is,

\begin{equation}
	\partial \delta_C(x) = N_C(x).
	\label{subgradientofindicator}
\end{equation}

\begin{proposition} \label{subgradientformula}
	Let \( C \subset \mathbb{R}^n \) be a nonempty, closed, and convex set. Then the subdifferential of the distance function \( d_C \) at a point \( a \in \mathbb{R}^n \) is given by:
	\[
	\partial d_C(a) = 
	\begin{cases}
		\left\{ \dfrac{a - \operatorname{proj}_C(a)}{\| a - \operatorname{proj}_C(a) \|} \right\}, & \text{if } a \notin C, \\[10pt]
		N_C(a) \cap \mathbb{B}, & \text{if } a \in C,
	\end{cases}
	\]
	where \( \mathbb{B} \) is the closed unit ball in \( \mathbb{R}^n \). The proof of this formula can be found, for example, in \cite[Example 3.3, p.~259]{HiLe96}.
\end{proposition}

\noindent
In particular, consider the case where the set \( C \subset \mathbb{R}^n \) is a singleton, i.e., \( C = \{b\} \) for some \( b \in \mathbb{R}^n \). Then, for any \( a \in \mathbb{R}^n \), the projection of \( a \) onto \( C \) is uniquely given by
\[
\operatorname{proj}_C(a) =\{ b\}.
\]
Thus, the distance from \( a \) to \( C \) simplifies to
\[
d_C(a) = \|a - b\|.
\]
In this case, the subdifferential of the distance function at \( a \), for \( a \neq b \), is given by
\begin{equation} \label{subgradientformula_forsingleton}
	\partial d_C(a) = \left\{ \frac{a - b}{\|a - b\|} \right\}.
\end{equation}

\begin{theorem}[Affine Composition Rule, {\cite[Theorem 4.2.1, p.263]{HiLe96}}]\label{affinetheorem}
	Let  $T: \mathbb{R}^n \to \mathbb{R}^m$ 
	be an affine mapping defined by $T(x) \;=\; T_0 \,x \;+\; c,$ where  	$T_0: \mathbb{R}^n \to \mathbb{R}^m$
	is a linear operator and  \(c \in \mathbb{R}^m\).  
	Let \(\psi: \mathbb{R}^m \to \mathbb{R}\) be a finite convex function. Then, for every \(x \in \mathbb{R}^n\),
	\[
	\partial\bigl(\psi\circ T\bigr)(x)
	\;=\;
	T_0^*\,\partial \psi\bigl(T(x)\bigr),
	\]
	where $T_0^*: \mathbb{R}^m \to \mathbb{R}^n$
	denotes the adjoint operator of \(T_0\). In other words, \(T_0^*\) is the unique linear map satisfying
	\[
	\langle T_0\,x,\; v \rangle
	\;=\;
	\langle x,\; T_0^*\,v \rangle
	\quad
	\forall\, x \in \mathbb{R}^n,\; v \in \mathbb{R}^m.
	\]
\end{theorem}

Following the discussion on \cite[p.\ 263]{HiLe96}, Let $f: \underbrace{\mathbb{R}^n \times \cdots \times \mathbb{R}^n}_{m\text{ times}}	\;\longrightarrow\; \mathbb{R}$
be a convex function in \(m\) blocks of variables, i.e.\ we write 
\((x_1, x_2, \dots, x_m)\) with each \(x_i \in \mathbb{R}^n\).  
Fix any \((x_2,\dots,x_m) \in \mathbb{R}^{n(m-1)}\), and define the affine map $ T: \mathbb{R}^n \to \mathbb{R}^{mn},\quad T(x_1) = (x_1, x_2, \dots, x_m)$ 
The linear part of \( T \) is \( T_0(x_1) = (x_1, 0, \dots, 0) \), and the corresponding adjoint is \( T_0^*(v_1, v_2, \dots, v_m) = v_1 \). Hence \( f \circ T \) is precisely the function \( x_1 \mapsto f(x_1, x_2, \dots, x_m) \).

\medskip

For each \(1 \le i \le m\), define the partial function
\[
f_i: \mathbb{R}^n \to \mathbb{R},
\quad
f_i(x_i) 
\;=\;
f\bigl(x_1,\dots,x_{i-1},\;x_i,\;x_{i+1},\dots,x_m\bigr),
\]
where the other \(x_j\) (with \(j \neq i\)) are held fixed. Applying~\ref{affinetheorem} to each such affine composition gives:
\[
\partial f_i(x_i)
\;=\;
\Bigl\{
v_i \in \mathbb{R}^n 
\;\Bigm|\;
\exists\,v_j \,(j\neq i)\text{ with }(v_1,\dots,v_m)\in \partial f(x_1,\dots,x_m)
\Bigr\}.
\]
Thus, \(\partial f_i(x_i)\) is the \emph{projection} of \(\partial f(x_1,\dots,x_m)\) onto the \(i\)th block. Consequently,
\begin{equation}\label{subset_inclusion}
	\partial f(x_1,\dots,x_m)
	\;\subseteq\;
	\partial f_1(x_1)
	\;\times\;
	\partial f_2(x_2)
	\;\times\;
	\cdots
	\;\times\;
	\partial f_m(x_m).
\end{equation}

\begin{remark}\label{remark:subgradient_equality}
	In general, the inclusion above in \ref{subset_inclusion}  is strict. However, if for each \(i\) the partial function \(f_i\) is differentiable at \(x_i\), then \(\partial f_i(x_i)\) reduces to a singleton \(\{\nabla f_i(x_i)\}\). Consequently, if all \(f_i\) are differentiable at the respective \(x_i\) for \(i = 1, \dots, m\), the subdifferential becomes as:
	\begin{equation}\label{sub_cartessian}
		\partial f(x_1,\dots,x_m)
		\;=\;
		\Bigl\{
		\bigl(\nabla f_1(x_1),\,\nabla f_2(x_2),\,\dots,\,\nabla f_m(x_m)\bigr)
		\Bigr\}.
	\end{equation}
\end{remark}

\begin{example}[Subdifferential of the pairwise distance]
	\label{ex:subdiff_distance}
	Let \( D(a_1, \dots, a_m) = \|a_1 - a_2\| \) denote the pairwise distance between \( a_1 \) and \( a_2 \) in \( \mathbb{R}^n \), $a_1 \ne a_2$ viewed as a function of the tuple \( (a_1, \dots, a_m) \in \mathbb{R}^{nm}\). The subdifferential of \( D \) with respect to the full tuple is given by the ordered tuple of partial subdifferentials:
	\[
	\partial D(a_1, \dots, a_m) = \bigl( \partial D_1(a_1),\, \partial D_2(a_2),\, \partial D_3(a_3),\, \dots,\, \partial D_m(a_m) \bigr),
	\]
	where \( \partial D_k(a_k) \) denotes the projection of \( \partial D(a_1, \dots, a_m) \) onto the \( k \)-th coordinate.
	
	By~\ref{remark:subgradient_equality} and the preceding discussion, we obtain:
	\[
	\partial D_1(a_1) = \frac{a_1 - a_2}{\|a_1 - a_2\|}, \qquad
	\partial D_2(a_2) = \frac{a_2 - a_1}{\|a_1 - a_2\|}, \qquad
	\partial D_k(a_k) = \mathbf{0} \quad \text{for all } k \ge 3.
	\]
	Applying the Cartesian product rule for subdifferentials see \eqref{sub_cartessian}, we conclude that
	\[
	\partial D(a_1, \dots, a_m) =
	\biggl(
	\frac{a_1 - a_2}{\|a_1 - a_2\|},\;
	\frac{a_2 - a_1}{\|a_1 - a_2\|},\;
	\mathbf{0}, \dots, \mathbf{0}
	\biggr).
	\]
	Thus, only the first two components contribute to the subdifferential while the remaining \( m - 2 \) components are zero.
\end{example}

Following is the most fundamental result regarding subdifferentials calculus and is widely used as sum rule.

\begin{theorem}[Moreau-Rockafellar Theorem {\cite[Corollary 3.21, p. 93]{MoNa23}}]\label{sumofsubgradient}
	Let \( \psi_i: \mathbb{R}^d \to (-\infty, +\infty] \), for \( i = 1, \dots, k \), be closed convex functions. Suppose there exists a point \( \bar{a} \in \bigcap_{i=1}^{k} \operatorname{dom} \psi_i \) at which all but at most one of the functions \( \psi_1, \dots, \psi_k \) are continuous. Then, for all \( a \in \bigcap_{i=1}^{k} \operatorname{dom} \psi_i \), the subdifferential of the sum satisfies the relation:
	\[
	\partial \left( \sum_{i=1}^{k} \psi_i \right)(a) = \sum_{i=1}^{k} \partial \psi_i(a) := \left\{ \sum_{i=1}^{k} w_i \mid w_i \in \partial \psi_i(a), \; i = 1, \dots, k \right\}.
	\]
\end{theorem}

\section{Optimality Condition}\label{Optimality_Conditions}
It is natural to ask whether problem~\eqref{objective} admits a solution. In this section, we show that a solution to the generalized problem indeed exists, even without assuming boundedness of the sets. We further establish an optimality condition that is both necessary and sufficient for the generalized waist problem~\eqref{objective}. This condition provides the theoretical foundation for the computational techniques presented in~\ref{Convergence}.

Let us begin with ensuring the existence of optimal solutions to \eqref{objective} under sufficient general conditions drawn from classical waist problem.

\begin{theorem}[\textbf{Existence}]\label{Existenceofsolution}
	The generalized waist problem~\ref{objective} admits an optimal solution if either all the sets \(C_1, \dots, C_m\) are bounded, or, in the case that whenever the set \(C_j\) is unbounded and points $a_j \in C_j$ with \(\|a_j\| \to \infty\) satisfies the condition that either $d_{C_{j-1}}(a_j) \to \infty$  or $d_{C_{j+1}}(a_j) \to \infty$ for all \(1 \leq j \leq m\).
\end{theorem}

\begin{proof}
	We prove the existence of a minimizer by considering two cases based on the boundedness of the sets $C_1, \dots, C_m$.
	\vspace{0.2cm}
	
	\paragraph{Case 1: Suppose all the sets $C_1, \dots, C_m$ are bounded}
	Since each $C_i$ is bounded and closed, their Cartesian product
	$C = C_1 \times C_2 \times \cdots \times C_m$ is bounded and closed in $\mathbb{R}^{mn}$, hence compact. Let $\{(a_1^{(k)}, \dots, a_m^{(k)})\}_{k=1}^\infty \subset C$ be a minimizing sequence, i.e.,
	\begin{equation}
		\lim_{k \to \infty} D(a_1^{(k)}, \dots, a_m^{(k)}) = \inf_{(a_1, \dots, a_m) \in C} D(a_1, \dots, a_m) =: \gamma_{\min}.
	\end{equation}
	Since $C$ is compact, by the Bolzano--Weierstrass theorem, there exists a subsequence $\{(a_1^{(k_r)}, \dots, a_m^{(k_r)})\}$ converging to some point $(a_1^*, \dots, a_m^*) \in C$. The objective function $D(a_1, \dots, a_m)$ is continuous (as it is composed of finite sums of continuous norms), thus passing the limit inside yields
	$	D(a_1^*, \dots, a_m^*) = \lim_{r \to \infty} D(a_1^{(k_r)}, \dots, a_m^{(k_r)}) = \gamma_{\min}.$ Therefore, $(a_1^*, \dots, a_m^*)$ is a global minimizer of $D$ over $C$.
	\vspace{0.3cm}
	
	\paragraph{Case 2:Suppose the set $C_j$ is unbounded}
	We can find points $a_j \in C_j$ with $\|a_j\| \to \infty$. By assumption, as $\|a_j\| \to \infty$, we have either
	$d(a_j, C_{j-1}) \to \infty$ or $d(a_j, C_{j+1}) \to \infty,$ which implies $\|a_j - a_{j-1}\| \to \infty ~ \forall a_{j-1} \in C_{j-1}$ or $\|a_j - a_{j+1}\| \to \infty, ~ \forall a_{j+1} \in C_{j+1}.$ Thus, the objective function $D(a_1, \dots, a_m)$ satisfies that $D(a_1, \dots, a_m) \to \infty$ when $\|(a_1, \dots, a_m)\| \to \infty$. That is, $D$ is coercive in this case. By the~\ref{coercivity_theorem}, coercivity of a continuous function ensures the existence of a minimizer. Thus, $D$ attains its minimum over $C$.	
	In both cases, we conclude that an optimal solution exists.
\end{proof}
\begin{remark}[\textit{Multiple Solutions in the generalized waist problem}]
	Under our proposed formulation of the generalized waist problem, consider the case of three parallel lines in $\mathbb{R}^3$ that are not all contained within the same plane (see~\ref{fig:parallel-lines}). This configuration leads to \emph{multiple solutions} to the minimization problem. Specifically, if we take a plane \( P \) that is perpendicular to one (and hence all) of the lines, the points where this plane intersects each line form an optimal solution.	
\end{remark}
In fact, every such intersection yields a triangle with the same total cyclic distance. Consequently, the set of optimal solutions consists of all such triplets of intersection points generated by varying the position of the perpendicular plane \( P \). This example illustrates that the uniqueness of the solution is not guaranteed without additional assumptions, such as strict convexity of the sets or non-parallel positioning. It is worth mentioning here that coercivity does not occur in this case; however, there are infinitely many optimal solutions
\begin{figure}[H]
	\centering
	\includegraphics[scale=0.15]{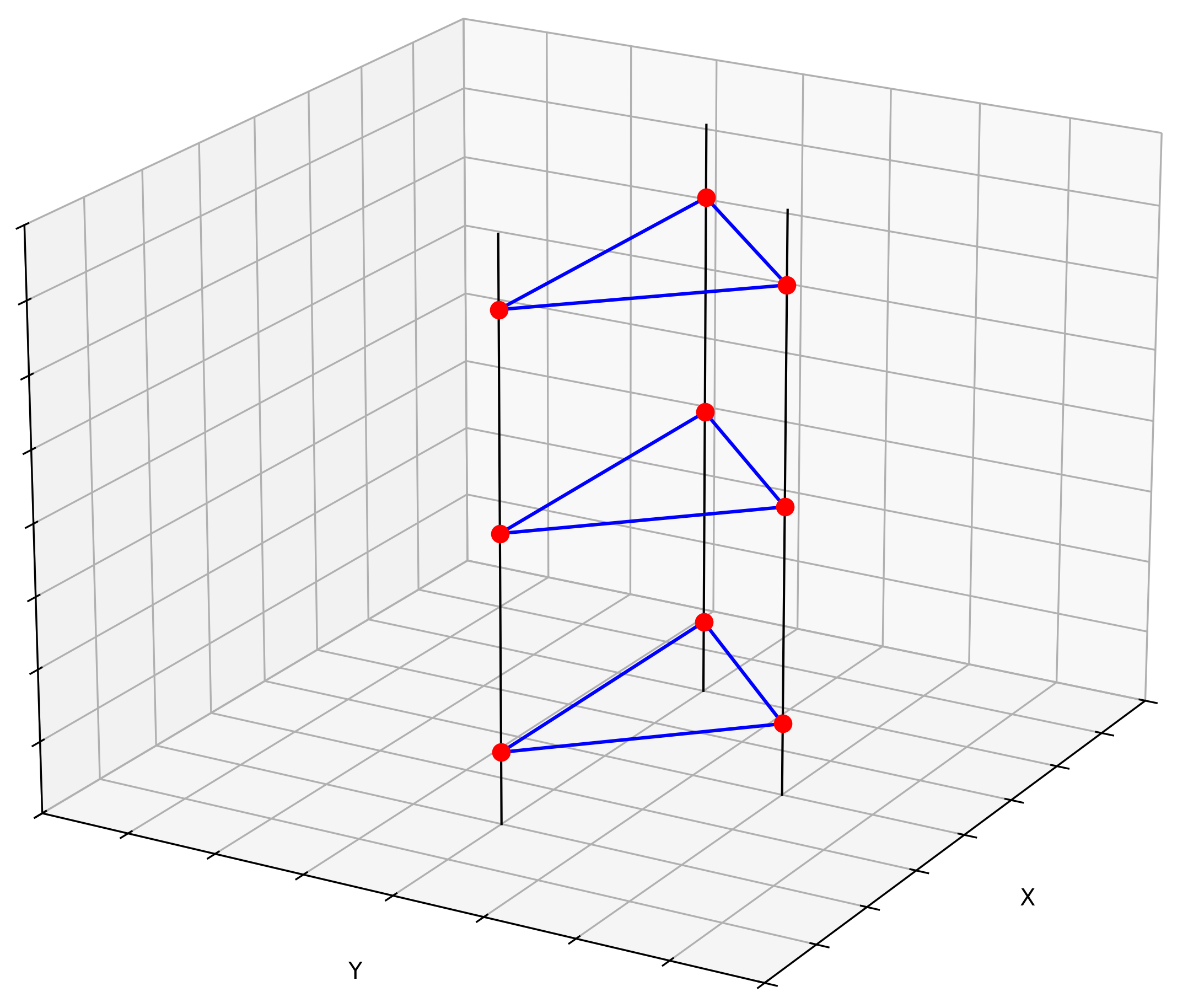}
	\caption{Three parallel lines not lying in the same plane. Multiple optimal solutions may occur.}
	\label{fig:parallel-lines}
\end{figure}
\begin{remark}
	Inspired by Exercise 1.1.29 of~\cite{AnMuSt07}, we construct a genuinely nontrivial example of the generalized waist problem 
	that admits \emph{infinitely many} solutions.
\end{remark}%

Enforcing the conditions on the four rectangles gives rise to the solution to be a parallelogram whose vertices are on the side of rectangles such that the law of reflection is followed on at least one vertex so that the law of reflection is followed automatically on other vertices.

\begin{figure}[H]
	\centering
	\includegraphics[scale=0.12]{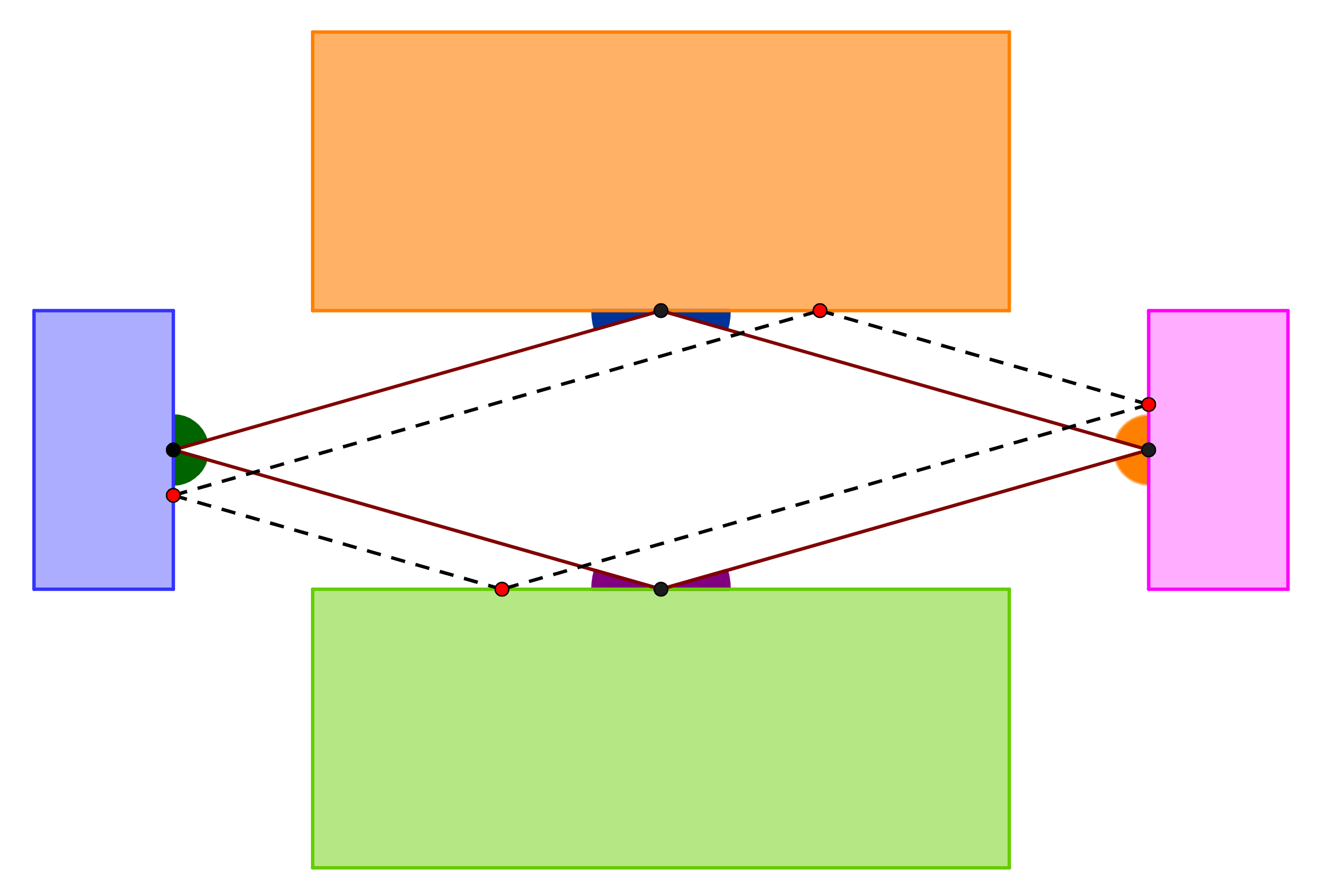}
	\caption{A Non-trivial example of generalized waist problem which has multiple solutions.}
	\label{fig:four_square}
\end{figure}

This naturally raises a fundamental question: \emph{under what mild conditions can we guarantee the uniqueness of the solution to the generalized waist problem?} Interestingly, it turns out that uniqueness can be ensured if the sets \(C_1, C_2, \dots, C_m\) are in \emph{general position} and at least one of them is strictly convex (see~\ref{Unique_Solution}).

To make this precise, we introduce the notion of general position for the sets \(C_i\):

\begin{definition}
	Let \(C_1, C_2, \dots, C_m \subset \mathbb{R}^n\) be nonempty, closed, convex sets. 
	For each \(i \in \{1, \dots, m\}\), define
	\[
	K_i := \mathrm{conv}\!\left(
	C_1 \cup \cdots \cup C_{i-1} \cup C_{i+1} \cup \cdots \cup C_m
	\right).
	\]
	We say that the collection \(\{C_i\}_{i=1}^m\) is in \emph{general position} if
	$C_i \cap K_i = \varnothing \quad \text{for all } i = 1, \dots, m.$
\end{definition}

With the structural condition of general position of the sets, we now turn to basic geometric observations that will help us to understand the boundary behavior of the feasible set.

\begin{remark}
	Let \( C_i \subseteq \mathbb{R}^{n} \) be nonempty convex sets for \( i = 1, \dots, m \). Define \( C = \prod_{i=1}^{m} C_i \).
	Then
	\[
	\bd(C) \supseteq \prod_{i=1}^{m} \bd(C_i),
	\]
	and this inclusion can be strict.
\end{remark}

\begin{example}
	Consider the convex sets \(C_1 = C_2 = [1,2] \subseteq \mathbb{R}\). Their Cartesian product is given by
	\[
	C = C_1 \times C_2 = [1,2] \times [1,2] = \{(x,y) \mid 1 \leq x \leq 2, \; 1 \leq y \leq 2\},
	\]
	which is a square in \(\mathbb{R}^2\).
	
	The boundary of each set \(C_i\) is	$\bd(C_i) = \{1,2\}, \quad i=1,2.$
	Thus,
	\[
	\bd(C_1) \times \bd(C_2) = \{1,2\} \times \{1,2\} = \{(1,1), (1,2), (2,1), (2,2)\}.
	\]
	However, the boundary of \(C\) itself includes all points along the four edges of the square. Clearly,
	\[
	\bd(C_1) \times \bd(C_2) = \{(1,1), (1,2), (2,1), (2,2)\} \subsetneq \bd(C).
	\]
	This explicitly illustrates that the inclusion
	\[
	\bd\left(\prod_{i=1}^{m} C_i\right) \supseteq \prod_{i=1}^{m} \bd(C_i)
	\]
	may hold strictly.
\end{example}

\begin{lemma}\label{interior_Distance_lemma}
	
	Suppose $C_i, i=1,2.\cdots, m$ are in general position and  $x = (x_1\allowbreak, x_2\allowbreak, \dots\allowbreak, x_m) \in C$ where  $C=C_1 \times C_2 \times \cdots \times C_m$. If there is at least one component \(x_i\in \interior(C_i)\), then there exists a point 
	$\bar{x} = (\bar{x}_1,\bar{x}_2,\dots,\bar{x}_m)\in \bd(C),$ with
	$\bar{x}_i\in \bd(C_i)$ for each $i=1,\dots,m,$	such that
	\[
	D(\bar{x}) \;<\; D(x).
	\]
\end{lemma}

\begin{proof}
	Given \(x = (x_1,x_2,\dots,x_m)\in C\), assume that at least one component \(x_i\) lies in \(\interior(C_i)\). We construct a point \(\bar{x} = (\bar{x}_1, \bar{x}_2, \dots, \bar{x}_m)\) by choosing its components \(\bar{x}_j\) as follows:
	\begin{enumerate}
		\item If \( x_j \in \bd(C_j) \), define \(\bar{x}_j = x_j\).
		\item If \( x_j \in \interior(C_j) \), select a boundary point \(\bar{x}_j \in \bd(C_j)\) such that
		\[
		\| \bar{x}_j - x_{j-1} \| + \|\bar{x}_j - x_{j+1} \| < \|x_j - x_{j-1} \| + \|x_j - x_{j+1} \|.
		\]
	\end{enumerate}
	We show that such a choice for \(\bar{x}_j\) can always be achieved. Without loss of generality assume that $x_1 \in \interior(C_1)$, consider the line segments joining $x_1$ to $x_2$ and $x_1$ to $x_m$. Let $y_1$ and $z_1$ be the boundary points of $C_1$ obtained on the extensions of the line segments $[x_1, x_2]$ and $[x_1, x_m]$, where we define \([x_1,x_2] := \{\lambda x_1 + (1-\lambda)x_2 \mid 0 \le \lambda \le 1\}\). Extending the line segment joining $x_1$ and $\frac{y_1+z_1}{2}$ must meet the boundary of $C_1$ at some point $\bar{x}_1$. Since $C_i$'s are in general position this also means that $C_{j-1}, C_j, C_{j+1}$ can be treated as vertices of a plane triangle, and note that $C_1 \cap \operatorname{conv}(C_2 \cup C_m)= \varnothing$, thus 	
	$\bar{x}_1$ is such that $\| \bar{x}_1 - x_{2} \|  + \|\bar{x}_1 - x_{m} \|< \|x_1 - x_{2} \| +\|x_1 - x_{m} \|$ ,see~\ref{fig:strict_Convex_lemmaimage} for geometric intuition.
	
	By this construction, every component \(\bar{x}_j\) of \(\bar{x}\) satisfies \(\bar{x}_j \in \bd(C_j)\). This ensures that \(\bar{x}\) is in $\bd(C)$.
	
	Since \(x\) has at least one component \(x_i \in \interior(C_i)\), at least one component is modified in the construction of \(\bar{x}\). The point \(\bar{x}_i \in \bd(C_i)\) is chosen specifically to reduce the distances to its neighbours \(x_{i-1}\) and \(x_{i+1}\). Because \(x_i\) is in \(\interior(C_i)\) and \(\bar{x}_i \in \bd(C_i)\) is obtained in the preceding manner, it is guaranteed that \(D(\bar{x}) < D(x)\).
\end{proof}%

\begin{remark}\label{Boundary_Remark}
If $C_j$ are in general position, the problem \eqref{objective}  is equivalent to
	
	\begin{equation}\label{Boundary_Objective}
		\min_{\bar{x} \in C} D(\bar{x}) 
		= \min_{\substack{\bar{x}_i \in \bd(C_i) \\ 1 \leq i \leq m}} D(\bar{x}_1, \bar{x}_2, \dots, \bar{x}_m)
	\end{equation}
\end{remark}

\begin{figure}[H]
	\centering
	\includegraphics[scale=0.12]{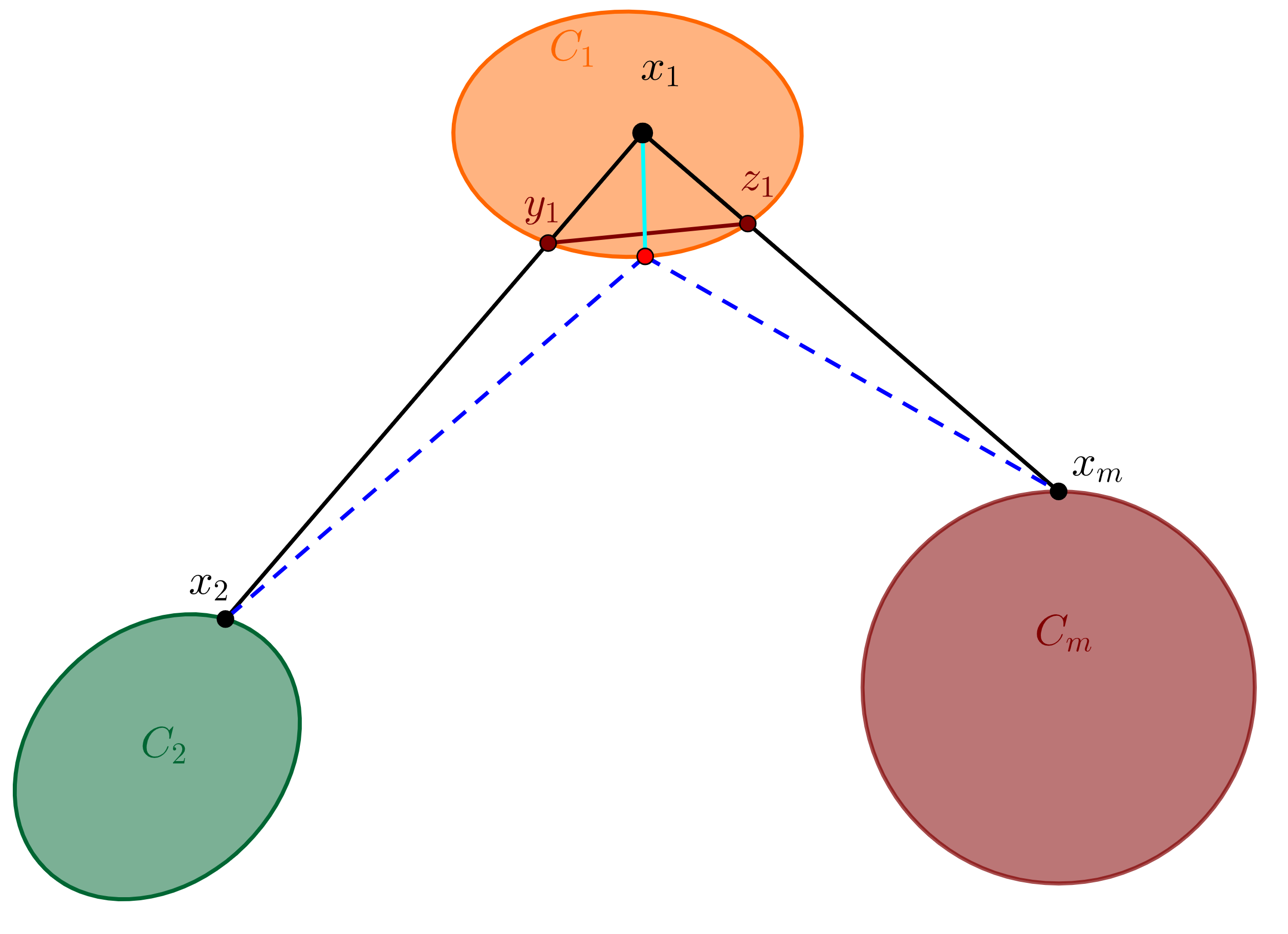}
	\caption{Minimum occur on the boundaries.}

	\label{fig:strict_Convex_lemmaimage}
\end{figure}

\begin{theorem}[\textbf{Uniqueness}]\label{Unique_Solution}
	Suppose the sets \(C_i\), for \(i = 1, 2, \dots, m\), are in general position. If all the sets \(C_i\) are strictly convex, then the generalized waist problem~\eqref{objective} admits a unique solution.
\end{theorem}

\begin{proof}
	Suppose, there exist two distinct optimal solutions \(X=(x_1,x_2,\dots,x_m)\) and \(Y=(y_1,y_2,\dots,y_m)\) with  optimal value $\gamma$. Clearly, by Remark \ref{Boundary_Remark}, each component of  $X$ and $Y$ are lying on the boundary of respective sets. Then we have,
	\[
	D(X) = D(Y) = \gamma.
	\]
	
Since $X \neq Y$, without loss of generality, assume $x_1 \neq y_1.$ Being  \(C_1\) is strictly convex, we have
	\[
	\frac{x_1+y_1}{2}\in\interior(C_1),
	\]
	Invoking Lemma~\ref{interior_Distance_lemma}, and strict convexity of the set $C_1$ implies:
	\begin{equation}\label{Dgeq}
		D\left(\frac{X+Y}{2}\right) > D(X) = D(Y) = \gamma.
	\end{equation}
	However, the convexity of the function \(D(\cdot)\) implies:
	\begin{equation}\label{Dleq}
		D\left(\frac{X+Y}{2}\right) \leq \frac{D(X)+D(Y)}{2} = \gamma.
	\end{equation}
	Clearly, the inequalities \(\eqref{Dgeq}\) and \(\eqref{Dleq}\) contradict each other, This ensures uniqueness of the optimal solution.
\end{proof}

We have now established both the existence and uniqueness of the generalized waist problem \eqref{objective}. Our next goal is to derive its optimality conditions. The theorem below provides necessary and sufficient conditions for a solution of \eqref{objective} to be optimal, laying the foundation for our numerical algorithms. As before, the concepts of subdifferentials and normal cones will be essential in formulating and proving these results.

\begin{theorem}\label{thm:optimality}
	Let $C_1, C_2, \ldots, C_m \subseteq \mathbb{R}^n$ be nonempty, closed, convex sets which are pairwise disjoint, that is  $C_i \cap C_j =\varnothing$ for $i \ne j$. Then a point 
	$a^* = (a_1^* \allowbreak, a_2^* \allowbreak, \dots \allowbreak, a_m^*) \in C$	is an optimal solution to \eqref{objective} if and only if there exist vectors 
	$n_i\in N_{C_i}(a_i^*)$  such that for each $i=1, \ldots, m$, we have
	\begin{equation}
		\frac{a_i^* - a_{i-1}^*}{\|a_i^* - a_{i-1}^*\|} + \frac{a_i^* - a_{i+1}^*}{\|a_i^* - a_{i+1}^*\|} + n_i = 0,
		\label{optimality_vector}
	\end{equation}
	where the indices are interpreted cyclically, that is, $a_0^*=a_m^*$ and $a_{m+1}^*=a_1^*$. Moreover, summing these equations from $i=1$ to $m$ yields
	\begin{equation}\label{normal_vector}
		n_1+n_2+\dots+n_m=0.
	\end{equation}
\end{theorem}

\begin{proof}
	
	We start by rewriting the constrained problem \eqref{objective} as the unconstrained minimization of 
	\begin{equation}
		F(a)=D(a)+\delta_C(a),
	\end{equation}
	By Fermat's rule (or the subgradient optimality condition) for convex functions, if $a^*$ minimizes $F(a)$ then
	\[
	0\in\partial F(a^*)=\partial\Bigl[D+\delta_C\Bigr](a^*).
	\]
	The sum rule for subdifferentials implies
	\[
	\partial\Bigl[D+\delta_C\Bigr](a^*)=\partial D(a^*)+\partial \delta_C(a^*).
	\]
	Moreover, it is well known from \eqref{subgradientofindicator} that 
	\[
	\partial\delta_C(a^*)=N_C(a^*),
	\]
	where $N_C(a^*)$ is the normal cone to $C$ at $a^*$. Since $
	C=C_1\times C_2\times\cdots\times C_m,$ we know(from \ref{Normalcone_Product}) $N_C(a^*)=N_{C_1}(a_1^*)\times N_{C_2}(a_2^*)\times\cdots\times N_{C_m}(a_m^*).$
	Hence, the optimality condition becomes
	\begin{equation}
		0\in \partial D(a^*)+ \Bigl( N_{C_1}(a_1^*)\times N_{C_2}(a_2^*)\times\cdots\times N_{C_m}(a_m^*) \Bigr).
		\label{eq:opt_cond1}
	\end{equation}
	
	We now compute the subdifferential \(\partial D(a^*)\) at an optimal point \(a^*\), where \(a^* = (a_1^*, \dots, a_m^*)\) and \(a^* \in C_1 \times \cdots \times C_m\). For a fixed index \(i\), the subgradient of \(\|a_i - a_{i+1}\|\) with respect to \(a_i\) is, by \ref{sumofsubgradient}, \ref{remark:subgradient_equality}, and \ref{ex:subdiff_distance}, given by

	\begin{align}
		\partial D(a^*) 
		&= \partial\left( \sum_{i=1}^{m} \lVert a_i - a_{i+1} \rVert \right) \nonumber\\
		&= \partial\left(
		\lVert a_1 - a_2 \rVert
		+ \lVert a_2 - a_3 \rVert
		+ \cdots
		+ \lVert a_m - a_1 \rVert
		\right) \nonumber\\
		&= \partial\left( \lVert a_1 - a_2 \rVert \right)
		+ \partial\left( \lVert a_2 - a_3 \rVert \right)
		+ \cdots
		+ \partial\left( \lVert a_m - a_1 \rVert \right) \nonumber\\
		&= \Biggl(
		\tfrac{a_1^* - a_2^*}{\lVert a_1^* - a_2^* \rVert},
		\tfrac{a_2^* - a_1^*}{\lVert a_2^* - a_1^* \rVert},
		0, \dots, 0
		\Biggr) \nonumber\\
		&\quad + \Biggl(
		0,
		\tfrac{a_2^* - a_3^*}{\lVert a_2^* - a_3^* \rVert},
		\tfrac{a_3^* - a_2^*}{\lVert a_3^* - a_2^* \rVert},
		0, \dots, 0
		\Biggr) \nonumber\\
		&\quad + \cdots \nonumber\\
		&\quad + \Biggl(
		\tfrac{a_1^* - a_m^*}{\lVert a_1^* - a_m^* \rVert},
		0, \dots, 0,
		\tfrac{a_m^* - a_1^*}{\lVert a_m^* - a_1^* \rVert}
		\Biggr).
		\label{eq:partial_subgradient_long}
	\end{align}

	Returning to \eqref{eq:opt_cond1}, we obtain that there exist vectors $n_i\in N_{C_i}(a_i^*)$, for $i=1,\dots,m$, such that
	\begin{align*}
		&\Biggl(
		\frac{a_1^* - a_m^*}{\|a_1^* - a_m^*\|}
		+ \frac{a_1^* - a_2^*}{\|a_1^* - a_2^*\|},\;
		\frac{a_2^* - a_1^*}{\|a_2^* - a_1^*\|}
		+ \frac{a_2^* - a_3^*}{\|a_2^* - a_3^*\|},\; \dots, \\
		&\qquad \frac{a_m^* - a_{m-1}^*}{\|a_m^* - a_{m-1}^*\|}
		+ \frac{a_m^* - a_1^*}{\|a_m^* - a_1^*\|}
		\Biggr)
		+ (n_1, n_2, \dots, n_m) = 0.
	\end{align*}

	This vector equation is equivalent to the system of equations
	\[
	\frac{a_i^*-a_{i-1}^*}{\|a_i^*-a_{i-1}^*\|}+\frac{a_i^*-a_{i+1}^*}{\|a_i^*-a_{i+1}^*\|}+ n_i =0,\quad i=1,\dots,m,
	\]
	where we recall the cyclic convention $a_0^*=a_m^*$ and $a_{m+1}^*=a_1^*$, which is exactly the statement in \eqref{optimality_vector}. Summing these equations over $i$ immediately gives
	$$ n_1+n_2+\dots+n_m=0. $$
	This completes the proof.
\end{proof}
\subsection{Geometrical interpretation of optimality conditions}
%\begin{remark}[Geometrical Interpretation of optimality conditions]
The equation \eqref{optimality_vector} reveals that the sum of unit vectors along $\overrightarrow{a_2a_1}$  and $\overrightarrow{a_ma_1}$ is equal to a normal vector to $C_1$ at $a_1$. This is equivalent to the statement that a normal vector $n_1$ to $C_1$ at $a_1$ is the bisector of the angle $\angle a_2a_1a_m.$\\

More interestingly, consider equation~\eqref{optimality_vector}. If $a_i^* \in \operatorname{int}(C_i)$, then the normal cone satisfies $N_{C_i}(a_i^*) = \{0\}$, which implies that $n_i = 0$. In this case, equation~\eqref{optimality_vector} simplifies to
\begin{equation}
	\frac{a_i^* - a_{i-1}^*}{\|a_i^* - a_{i-1}^*\|} = - \frac{a_i^* - a_{i+1}^*}{\|a_i^* - a_{i+1}^*\|}
\end{equation}
This relation indicates that the vectors $\overrightarrow{a_{i-1}a_i}$ and $\overrightarrow{a_{i+1}a_i}$ are collinear and point in opposite directions. In other words, the points $a_{i-1}$, $a_i$, and $a_{i+1}$ must lie on a straight line. However, when the sets $C_i$ are in general position, such collinearity is generally not possible. Therefore, it is noteworthy that under the general position assumption, equation~\eqref{optimality_vector} implies that each optimal point $a_i^*$ must lie on the boundary of its corresponding set $C_i$.

The equation \eqref{normal_vector} verbally says that the normal vectors, which are the angle bisectors of the minimal chain (the closed polygonal chain having the minimal perimeter), sum to zero. It indicates towards a point $x \in \mathbb{R}^n$ such that all the angle bisectors of the minimal chain meet at $x$. Moreover, the equation \eqref{optimality_vector}  and \eqref{normal_vector} suggest that the point $x$ must be in the relative interior of the convex hull of the closed polygonal chain (convex polytope) $(a_1, a_2, \cdots, a_m)$.
In other words, when $C_i$'s are in general position, then there is a point $x$ (unique if atleast one $C_i$ is strictly convex), which is the intersection of all the bisectors of the closed polygonal chain which has the vertices as the euclidean projection of $x$ on the given convex sets, in this case the normal vector  $n_i= x-\operatorname{proj}_{C_i}(x) \in N_{C_i}(\operatorname{proj}_{C_i}(x))$.

\section{Convergence Analysis of a Numerical Algorithm}\label{Convergence}
In this section, based on the optimality conditions, we frame the computational method and rigorously examine the conditions under which the algorithm converges to an optimal solution.

\begin{theorem}[\textbf{Convergence of Projected Subgradient Descent for the generalized waist problem}]
	Let $C_1, C_2, \dots, C_m \subseteq\mathbb{R}^n$ be non-empty, closed, and convex sets, which are pairwise disjoint. Assume that boundedness condition and growth conditions of \ref{Existenceofsolution} are satisfied. 
	Consider the sequence \( \{(a_1^{(k)},\allowbreak\dots,\allowbreak a_m^{(k)})\} \)
	generated by the \textit{Projected Subgradient Descent Algorithm}, defined as:

	\begin{align}
		a_i^{(k+1)}
		= \operatorname{proj}_{C_i} \Bigl( a_i^{(k)}
		- \alpha_k \nabla_{a_i} D^{(k)} \Bigr),
		\quad i = 1, \dots, m.
		\label{eq:update_rules}
	\end{align}

	where \( \operatorname{proj}_{C_i}(x) \) denotes the projection operator onto the convex set \( C_i \), and \( \nabla_{a_i} D^{(k)} \) are the subgradients of the objective function \( D \)  with respect to $a_i$ at iteration \( k \). The step size sequence \( \{\alpha_k\} \) satisfies following conditions:
	
	\begin{equation} \label{stepsizecondition}
		\sum_{k=1}^\infty \alpha_k = \infty \quad \text{and} \quad \sum_{k=1}^\infty \alpha_k^2 < \infty.
	\end{equation}

	Then, the following results hold:
	\begin{enumerate}
		\item The sequence \((a_1^{(k)}\allowbreak, a_2^{(k)}\allowbreak, \dots\allowbreak, a_m^{(k)})\) converges to an optimal solution \((a_1^*\allowbreak, a_2^*\allowbreak, \dots\allowbreak, a_m^*)\) of the optimization problem~\eqref{objective}.
		\item The sequence of objective values \(D^{(k)} = D(a_1^{(k)}\allowbreak, a_2^{(k)}\allowbreak, \dots\allowbreak, a_m^{(k)})\) converges to the global minimum \(D^*\).
	\end{enumerate}
	
\end{theorem}
\begin{proof}
	The algorithm described in \eqref{eq:update_rules} is well-posed because the projection onto a convex set is uniquely defined. Specifically, the projection operator \( \operatorname{proj}_{C_i}(x) \) ensures that the updated points \( (a_1^{(k+1)}, a_2^{(k+1)}, \dots,a_m^{(k+1)}) \) remain within their respective convex sets \( C_1, C_2, \dots,C_m \). 
	Since the sets \( C_1, C_2, \dots,C_m \) are closed and convex, the optimization problem admits an optimal solution, as guaranteed by~\ref{Existenceofsolution}. The convergence of the algorithm under the step size conditions \eqref{stepsizecondition} specified in algorithm  \eqref{eq:update_rules} is based on the theory of the subgradient method for convex functions. This is specifically the ``square summable but not summable" case, as established in \cite[Page 480, Proposition 8.2.6]{BeNeOz03}.
	Additionally, the subgradient sum rule theorem, as stated in~\ref{sumofsubgradient}, and the subgradient formula for the distance function, given in~\ref{subgradientformula}, are critical to the analysis. These results ensure that the subgradients computed at each iteration are valid and that the sequence of updates converges to a stationary point. Since \( D(a_1, a_2, \dots,a_m) \) is convex, this stationary point corresponds to the global minimum of the objective function.
	Thus, the sequence \(\{a_1^{(k)}\allowbreak, a_2^{(k)}\allowbreak, \dots\allowbreak, a_m^{(k)}\}\) converges to the optimal solution \((a_1^*\allowbreak, a_2^*\allowbreak, \cdots\allowbreak, a_m^*)\), and the corresponding objective values \(D^{(k)}\) converge to the global minimum \(D^*\).	
\end{proof}

It is necessary to emphasize that the efficiency and convergence heavily rely on the step size. The following gives a nice charaterization to choose a step size for the proposed algorithm. 

\begin{theorem}[\textbf{Selection of optimal step size using exact line search method}]\label{exactlinesearchtheorem}
	Let $f: \mathbb{R}^n \to \mathbb{R}$ be defined as $f(a_1) = \|a_1 - a_2^*\| + \|a_1 - a_3^*\|$, where $a_2^*, a_3^* \in \mathbb{R}^n$ are fixed points. Let $C \subseteq \mathbb{R}^n$ be a nonempty closed convex set. Consider the constraint  optimization problem:
	$$
	\min_{a_1 \in C} f(a_1).
	$$
	Given an iterate $a_1^{(k)} \in C$, let $d^{(k)} \in \partial f(a_1^{(k)})$ be a subgradient of $f$ at $a_1^{(k)}$. Define the descent direction as $p^{(k)} = -d^{(k)}$. The next iterate $a_1^{(k+1)}$ is obtained by:
	$$
	a_1^{(k+1)} = a_1^{(k)} + \alpha p^{(k)},
	$$
	where the optimal step size $\alpha$ is chosen to minimize the objective function along the direction $p^{(k)}$. To compute $\alpha$, let
	$$
	\phi(\alpha) = f(a_1^{(k)} + \alpha p^{(k)}) =  \|a_1^{(k)} + \alpha p^{(k)} - a_2^*\| + \|a_1^{(k}) + \alpha p^{(k)} - a_3^*\|.
	$$
	Define 
	$$
	t_1(\alpha) = a_1^{(k)} + \alpha p^{(k)} - a_2^*, \quad t_2(\alpha) = a_1^{(k)} + \alpha p^{(k)} - a_3^*.
	$$
	Then,
	the optimal step size $\alpha$ is a solution to the equation:
	$$
	\phi'(\alpha) = \frac{t_1(\alpha)^\top}{\|t_1(\alpha)\|} p^{(k)} + \frac{t_2(\alpha)^\top}{\|t_2(\alpha)\|} p^{(k)} = 0.
	$$
	(The above equation can be solved numerically to obtain $\alpha$.)
\end{theorem}
\begin{proof}
	The optimization problem is to find an $a_1 \in C$ that minimizes the function $f(a_1) = \|a_1 - a_2^*\| + \|a_1 - a_3^*\|$, where $a_2^*, a_3^* \in \mathbb{R}^n$ are fixed. We aim to find a method to iteratively minimize the function.
	
	At iteration $k$, given the current iteration $a_1^{(k)}$, we take a step along the descent direction $p^{(k)}=-d^{(k)}$ where $d^{(k)}$ is a subgradient of $f$ at $a_1^{(k)}$.
	The corresponding iteration scheme is given by:
	$$a_1^{(k+1)} = a_1^{(k)} + \alpha p^{(k)}.	$$	
	The optimal step size $\alpha$ is determined by minimizing $f$ along the direction $p^{(k)}$. Hence,
	$$	\alpha = \arg\min_{\alpha \geq 0} \phi(\alpha),	$$
	where 
	$$
	\phi(\alpha) = f(a_1^{(k)} + \alpha p^{(k)}) = \|a_1^{(k)} + \alpha p^{(k)} - a_2^*\| + \|a_1^{(k)} + \alpha p^{(k)} - a_3^*\|.
	$$
	We seek $\alpha$ that makes the derivative of $\phi$ equal to 0.
	Let $t_1(\alpha) = a_1^{(k)} + \alpha p^{(k)} - a_2^*$ and $t_2(\alpha) = a_1^{(k)} + \alpha p^{(k)} - a_3^*$. Then, we can rewrite $\phi(\alpha)$ as 
	$$
	\phi(\alpha) = \|t_1(\alpha)\| + \|t_2(\alpha)\|.
	$$
	Taking the derivative with respect to $\alpha$, we get:
	$$
	\phi'(\alpha) = \frac{t_1(\alpha)^\top}{\|t_1(\alpha)\|} \frac{d}{d\alpha} t_1(\alpha) + \frac{t_2(\alpha)^\top}{\|t_2(\alpha)\|} \frac{d}{d\alpha} t_2(\alpha).
	$$
	Since $t_1(\alpha) = a_1^{(k)} + \alpha p^{(k)} - a_2^*$ and $t_2(\alpha) = a_1^{(k)} + \alpha p^{(k)} - a_3^*$, we have $\frac{d}{d\alpha} t_1(\alpha) = p^{(k)}$ and $\frac{d}{d\alpha} t_2(\alpha) = p^{(k)}$. Therefore,
	$$
	\phi'(\alpha) = \frac{t_1(\alpha)^\top}{\|t_1(\alpha)\|} p^{(k)} + \frac{t_2(\alpha)^\top}{\|t_2(\alpha)\|} p^{(k)}.
	$$
	The optimal step size $\alpha$ is found by solving equation $\phi'(\alpha) = 0$, or equivalently:
	$$
	\frac{t_1(\alpha)^\top}{\|t_1(\alpha)\|} p^{(k)} + \frac{t_2(\alpha)^\top}{\|t_2(\alpha)\|} p^{(k)} = 0.
	$$	
\end{proof}

Due to the complexity of the derivative, finding $\alpha$ analytically can be difficult. Instead, $\alpha$ can be approximated using numerical methods.

Understanding how rapidly an algorithm approaches its solution is essential for evaluating its efficiency. Convergence rates are typically classified as linear, sublinear, or superlinear and they  correspond respectively to geometric, polynomial, or faster-than-geometric decay of the error. We now recall these definitions from \cite{OR00}.

\begin{definition}[Linear Convergence]
	\label{def:linear_convergence}
	Let $\{a_k\}$ be a sequence converging to $a$, and define the error $e_k = \|a_k - a\|$.  We say $\{a_k\}$ converges \emph{linearly} to $a$ if there exist constants $c\in(0,1)$ and $k_0\in\mathbb{N}$ such that
	\[
	\|a_{k+1}-a\|\;\le\;c\,\|a_k-a\|
	\quad\forall\,k\ge k_0.
	\]
	Equivalently,
	\[
	\lim_{k\to\infty}\frac{e_{k+1}}{e_k} = r,\quad 0<r<1.
	\]
\end{definition}

\begin{definition}[Sublinear Convergence]
	\label{def:sublinear_convergence}
	A sequence $\{a_k\}$ converging to $a$ has \emph{sublinear convergence} if its error $e_k=\|a_k-a\|$ decays more slowly than any geometric rate.  A common characterization is
	\[
	e_k = \mathcal{O}\left(\frac{1}{k}\right),
	\quad\text{or equivalently}\quad
	\lim_{k \to \infty}\frac{e_{k+1}}{e_k}=1.
	\]
\end{definition}

\begin{definition}[Superlinear Convergence]
	\label{def:superlinear_convergence}
	A sequence $\{a_k\}$ converges \emph{superlinearly} to $a$ if
	\[
	\lim_{k\to\infty}\frac{e_{k+1}}{e_k} = 0.
	\]
\end{definition}

\subsection{Accelerating Convergence}
\label{Accelerating_Convergence}

The projected subgradient method theoretically converges at a sublinear rate, which can be slow in practice. To accelerate this behavior, we apply Aitken’s $\Delta^2$ transform to the iterates. In this subsection, we show that this classical sequence‐transformation systematically boosts convergence.

We illustrate our approach to solve the generalized Waist problem  \eqref{objective}, using the projected subgradient method~ \ref{alg:psd_algorithm}. The  standard analysis establishes a sublinear rate \cite[Theorem 9.16]{Be14}, namely under diminishing step sizes, although these bounds guarantee convergence, they can be slow when high accuracy is required.  Empirically, however, the iterates often exhibit locally linear behavior:
\[
\frac{\|a_{k+1}-a\|}{\|a_k-a\|}\;\to\;\mu\in(0,1),
\quad k\to\infty.
\]
To accelerate convergence, we employ Aitken’s $\Delta^2$ transformation defined by:
\[
\tilde{a}_k
\;=\;
a_k
\;-\;
\frac{(a_{k+1}-a_k)^2}
{\,a_{k+2}-2a_{k+1}+a_k\,}.
\]
The resulting transformed sequence of iterates, $\tilde{a}_k$, converges superlinearly to the limit $a$. To establish this result, we leverage the sublinear convergence property of the original sequence $a^{(k)}$.

The proof of following theorem can be found in any standard textbook of numerical analysis, e.g. see \cite{Po17}. To make the paper self contained we are representing it again.
\begin{theorem}[Aitken’s \(\Delta^2\) Transform Accelerates Convergence]\label{thm:aitken}
	Let \(\{a_k\}\) be a sequence converging to \(a\). Suppose that, for sufficiently large \(k\), the ratio of consecutive errors approaches a constant \(\mu \in (0,1)\), namely
	\[
	\lim_{k \to \infty} \frac{a_{k+1} - a}{a_k - a} = \mu.
	\]
	This condition implies that \(\{a_k\}\) converges approximately linearly to \(a\). Define the \emph{Aitken-\(\Delta^2\) transform} of \(\{a_k\}\) by
	\[
	\tilde{a}_k = a_k - \frac{(a_{k+1} - a_k)^2}{a_{k+2} - 2a_{k+1} + a_k}.
	\]
	Under the assumption that the ratio above is strictly constant, the transformed sequence \(\{\tilde{a}_k\}\) converges \emph{superlinearly} to \(a\). More precisely,
	\[
	\lim_{k \to \infty} \frac{\tilde{a}_{k} - a}{a_{k+2} - a} = 0.
	\]
\end{theorem}%
\begin{proof}
	Consider the sequence \(\{a_k\}\) converging to \(a\). By hypothesis, there exists a constant \(\mu \in (0,1)\) such that
	\[
	\lim_{k \to \infty} \frac{a_{k+1} - a}{a_k - a} = \mu.
	\]
	In particular, this implies
	\[
	\lim_{k \to \infty} \frac{a_{k+2} - a}{a_k - a} = \mu^2.
	\]
	We define the Aitken-\(\Delta^2\) transform of \(\{a_k\}\) by
	\[
	\tilde{a}_k = a_k - \frac{(a_{k+1} - a_k)^2}{a_{k+2} - 2a_{k+1} + a_k}.
	\]
	To check its convergence, consider the ratio
	\[
	\frac{\tilde{a}_k - a}{a_{k+2} - a}
	= \frac{(a_k - a) - \dfrac{(a_{k+1} - a_k)^2}{a_{k+2} - 2a_{k+1} + a_k}}{a_{k+2} - a},
	\]
	which simplifies to
	\begin{equation}\label{eq:aitken_ratio}
		\frac{\tilde{a}_k - a}{a_{k+2} - a}
		= \frac{a_k - a}{a_{k+2} - a}
		\left(
		1 - \frac{\left( \tfrac{a_{k+1} - a}{a_k - a} - 1 \right)^2}{
			\tfrac{a_{k+2} - a}{a_k - a} - 2\tfrac{a_{k+1} - a}{a_k - a} + 1}
		\right).
	\end{equation}
	Taking the limit as \(k \to \infty\) and using the earlier limits, we obtain
	\[
	\lim_{k \to \infty} \frac{\tilde{a}_{k} - a}{a_{k+2} - a}
	= \frac{1}{\mu^2}
	\left(
	1 - \frac{(\mu - 1)^2}{\mu^2 - 2\mu + 1}
	\right)
	= 0.
	\]
\end{proof}%

\section{Numerical Illustration}\label{Numericalexample}

This section presents the numerical solution for the proposed \textit{generalized waist problem.} Solving the \textit{ generalized waist problem \eqref{objective}} requires efficient numerical methods. The projected subgradient algorithm is a natural choice for a class of problems involving distance minimizations due to its simplicity and its ability to handle convex constraints by projecting onto the sets \(C_1, C_2, \cdots C_m\). The algorithm and technique are presented and illustrated with examples in \(\mathbb{R}^2\) and \(\mathbb{R}^3\). Furthermore, we analyze the performance of the algorithm and compare its convergence with the Aitken's \(\Delta^2\) Acceleration, which is used to accelerate convergence.

To make the algorithm easier to understand and implement, we present the following algorithm. This provides a clear and straightforward outline of the main computational steps, offering an intuitive overview of the method.

\begin{algorithm}[H]
	\caption{Projected Subgradient Descent for the generalized waist problem}
	\label{alg:psd_algorithm}
	\begin{algorithmic}[1]
		\REQUIRE Convex sets $C_1,\dots,C_m$, initial points $a_i^{(0)}\in C_i$ for $i=1,\dots,m$, step-size sequence $\{\alpha_k\}$, tolerance $\epsilon > 0$
		\ENSURE Optimal tuple $(a_1^*,\dots,a_m^*)$ and minimum perimeter $D^*$
		
		\STATE $k \gets 0$, $D_{\mathrm{prev}} \gets +\infty$
		\REPEAT
		\FOR{$i = 1$ \TO $m$}
		\STATE Let $i^- = i - 1$ and $i^+ = i + 1$ (with cyclic indexing: $0 \mapsto m,\ m+1 \mapsto 1$)
		\STATE Compute subgradient
		\[
		g_i^{(k)} := \frac{a_i^{(k)} - a_{i^-}^{(k)}}{\|a_i^{(k)} - a_{i^-}^{(k)}\|}
		+ \frac{a_i^{(k)} - a_{i^+}^{(k)}}{\|a_i^{(k)} - a_{i^+}^{(k)}\|}
		\]
		\ENDFOR
		\FOR{$i = 1$ \TO $m$}
		\STATE Update
		\[
		a_i^{(k+1)} := \operatorname{proj}_{C_i}\left(a_i^{(k)} - \alpha_k g_i^{(k)}\right)
		\]
		\ENDFOR
		\STATE Compute objective
		\[
		D^{(k+1)} := \sum_{i=1}^m \|a_i^{(k+1)} - a_{i^+}^{(k+1)}\|
		\]
		\IF{$|D^{(k+1)} - D_{\mathrm{prev}}| < \epsilon$}
		\STATE \textbf{break}
		\ENDIF
		\STATE $D_{\mathrm{prev}} \gets D^{(k+1)}$, $k \gets k + 1$
		\UNTIL{convergence}
		
		\RETURN $(a_1^*, \dots, a_m^*) = (a_1^{(k+1)}, \dots, a_m^{(k+1)})$, $D^* = D^{(k+1)}$
	\end{algorithmic}
\end{algorithm}

\begin{example}\label{disc_example}
	We consider the \emph{generalized waist problem} for three convex discs in the Euclidean plane. Each disc is defined with its corresponding center and radius:
	
	\[
	\begin{aligned}
		C_1 &= \{(x,y) \in \mathbb{R}^2 \mid (x-2)^2 + (y-3)^2 \leq 1\}, \quad \text{center } (2,3), \text{ radius } =1; \\
		C_2 &= \{(x,y) \in \mathbb{R}^2 \mid (x-8)^2 + (y-4)^2 \leq 4\}, \quad \text{center } (8,4), \text{ radius } =2; \\
		C_3 &= \{(x,y) \in \mathbb{R}^2 \mid (x-4)^2 + (y-11)^2 \leq 9\}, \quad \text{center } (4,11), \text{ radius } =3.
	\end{aligned}
	\]
	
	The objective is to find one point from each disc, \(a_1 \in C_1\), \(a_2 \in C_2\), and \(a_3 \in C_3\), such that the total perimeter formed by connecting these points is minimized. In this case objective function is 
	$$	D(a_1, a_2, a_3):= \|a_1 - a_2\| + \|a_2 - a_3\| + \|a_3 - a_1\|.$$
\end{example}
To solve this optimization problem, the Projected Subgradient Descent (PSD) algorithm, as outlined in \ref{alg:psd_algorithm}, was employed. A MATLAB program was developed to implement the algorithm specifically for the given problem. We choose carefully  initial points within each disc: \(a_1^{(0)} = (1,3) \in C_1\), \(a_2^{(0)} = (10,4) \in C_2\), and \(a_3^{(0)} = (1,11) \in C_3\). The step size parameter \(\alpha_k\) was determined to be \textbf{$2.0707749$} through an exact line search, as described in \ref{exactlinesearchtheorem}.

The PSD algorithm successfully converges to an approximate optimal solution, yielding the following points on their respective discs:
$a_1^* = (2.7231,3.6907), \quad a_2^* = (6.1404,4.7362), \quad a_3^* = (4.2653,8.0118).$ These points minimize the total perimeter formed by connecting them, resulting in an optimal value of \(D^* = 11.9359\). This result highlights that the effectiveness of the PSD algorithm in solving the \textit{generalized waist problem} for the specified configuration of convex discs. \ref{fig:optimal_pointsdisc1} visually represents the discs \(C_1\), \(C_2\), and \(C_3\), 
the optimal points \((a_1^*,\allowbreak\ a_2^*,\allowbreak\ a_3^*)\), and the connecting line segments, 
as shown in \ref{fig:optimal_pointsdisc1}.  A detailed iteration-wise analysis is provided in \ref{tab:example1_iteration_data}, 
including the sequence of points \(a_1\allowbreak\), \(a_2\allowbreak\), and \(a_3\allowbreak\), 
the total distance \(D_k\allowbreak\), the change in total distance \(\Delta D_k\allowbreak\), and the convergence rate \(\xi\).
Together, these elements illustrate the development of the algorithm and its convergence to the optimal solution. Furthermore, to accelerate the convergence of the sequence of points, we implemented Aitken's acceleration as discussed in \ref{thm:aitken}. As shown in \ref{tab:example1_aitken_data}, Aitken's acceleration method reduces the number of iterations to 12 required for convergence, compared to 19 iterations  required for convergence in the unaccelerated PSD algorithm.

\begin{figure}[H]
	\centering
	\makebox[\textwidth]{\includegraphics[width=1.9\textwidth]{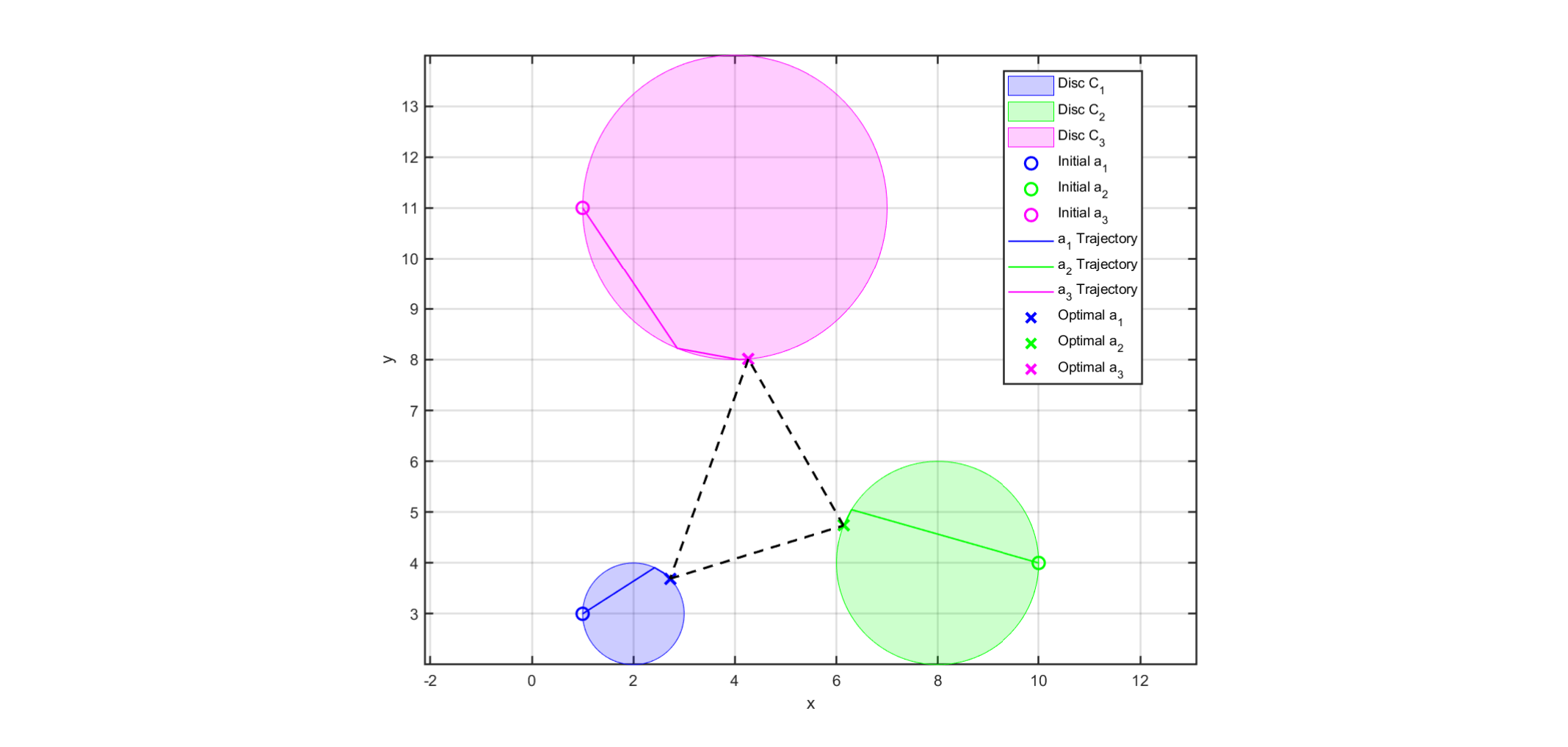}}
	\caption{Optimal Points $a_1^*$, $a_2^*$, and $a_3^*$ on Convex Discs with Minimized Sum of Distances.}
	\label{fig:optimal_pointsdisc1}
\end{figure}

\begin{table}[H]
	\centering
	\begin{sideways}
		\renewcommand{\arraystretch}{1.5}
		\begin{tabular}{|c|c|c|c|c|c|c|}
			\hline
			$k$ & $a_1 (x,y)$ & $a_2 (x,y)$ & $a_3 (x,y)$ & $D_k$ & $\Delta D_k$  & $\xi=\varepsilon_{k+1}/\varepsilon_{k}$ \\ 
			\hline
			1 & $(1.00000, 3.00000)$ & $(10.00000, 4.00000)$ & $(1.00000, 11.00000)$ & 28.45714 & - & - \\ 
			2 & $(2.41790, 3.90849)$ & $(6.29734, 5.04926)$ & $(2.86458, 8.22316)$ & 13.05662 & 15.400521476732889 & 1.00958 \\ 
			3 & $(2.59322, 3.80504)$ & $(6.11771, 4.67601)$ & $(4.08232, 8.00113)$ & 11.98162 & 1.075001604740480 & 0.99824 \\ 
			4 & $(2.71170, 3.70249)$ & $(6.13740, 4.72852)$ & $(4.20978, 8.00734)$ & 11.93761 & 0.044003727262275 & 0.99802 \\ 
			5 & $(2.72017, 3.69380)$ & $(6.13800, 4.73004)$ & $(4.25686, 8.01102)$ & 11.93601 & 0.001597725991488 & 0.99951 \\ 
			6 & $(2.72297, 3.69088)$ & $(6.13993, 4.73494)$ & $(4.26283, 8.01154)$ & 11.93595 & 0.000066006792645 & 0.99991 \\ 
			7 & $(2.72309, 3.69076)$ & $(6.14026, 4.73577)$ & $(4.26485, 8.01171)$ & 11.93595 & 0.000003384423957 & 0.99998 \\ 
			8 & $(2.72315, 3.69069)$ & $(6.14040, 4.73612)$ & $(4.26519, 8.01174)$ & 11.93595 & 0.000000202561294 & 0.99999 \\ 
			9 & $(2.72315, 3.69070)$ & $(6.14043, 4.73620)$ & $(4.26530, 8.01175)$ & 11.93595 & 0.000000013087915 & 0.99999 \\ 
			10 & $(2.72315, 3.69069)$ & $(6.14044, 4.73622)$ & $(4.26532, 8.01176)$ & 11.93595 & 0.000000000873264 & 0.99999  \\ 
			11 & $(2.72315, 3.69069)$ & $(6.14044, 4.73623)$ & $(4.26533, 8.01176)$ & 11.93595 & 0.000000000059005 & 0.99999 \\ 
			12 & $(2.72315, 3.69069)$ & $(6.14044, 4.73623)$ & $(4.26533, 8.01176)$ & 11.93595 & 0.000000000004009 & 0.99999 \\ 
			13 & $(2.72315, 3.69069)$ & $(6.14044, 4.73623)$ & $(4.26533, 8.01176)$ & 11.93595 & 0.000000000000270 & 0.99999  \\ 
			14 & $(2.72315, 3.69069)$ & $(6.14044, 4.73623)$ & $(4.26533, 8.01176)$ & 11.93595 & 0.000000000000020 & 0.99999 \\ 
			15 & $(2.72315, 3.69069)$ & $(6.14044, 4.73623)$ & $(4.26533, 8.01176)$ & 11.93595 & 0.000000000000004 & 0.99999  \\ 
			16 & $(2.72315, 3.69069)$ & $(6.14044, 4.73623)$ & $(4.26533, 8.01176)$ & 11.93595 & 0.000000000000002 & 0.99999  \\ 
			17 & $(2.72315, 3.69069)$ & $(6.14044, 4.73623)$ & $(4.26533, 8.01176)$ & 11.93595 & 0.000000000000002 & 0.99999  \\ 
			18 & $(2.72315, 3.69069)$ & $(6.14044, 4.73623)$ & $(4.26533, 8.01176)$ & 11.93595 & 0.000000000000002 & 0.99999  \\ 
			19 & $(2.72314, 3.69069)$ & $(6.14043, 4.73623)$ & $(4.26532, 8.01175)$ & 11.93595 & 0.000000000000000 & -  \\
			\hline
		\end{tabular}
	\end{sideways}
	\caption{Iteration-wise optimal points and distances for the generalized waist problem with three discs.}
	\label{tab:example1_iteration_data}
\end{table}

\begin{table}[H]
	\centering
	\begin{sideways}
		\renewcommand{\arraystretch}{1.5}
		\begin{tabular}{|c|c|c|c|c|c|c|}
			\hline
			$k$ & $a_1 (x,y)$ & $a_2 (x,y)$ & $a_3 (x,y)$ & $D_k$ & $\Delta D_k$  & $\xi=\varepsilon_{k+1}/\varepsilon_{k}$ \\ 
			\hline
			1 & $(2.417901, 3.908492)$ & $(6.297342, 5.049264)$ & $(2.864581, 8.223163)$ & 13.056618 & - &- \\ 
			2 & $(2.603894, 3.797065)$ & $(6.148965, 4.757410)$ & $(4.082319, 8.001130)$ & 11.981616 & -1.075001604740480 & 0.2007 \\  
			3 & $(2.706840, 3.707373)$ & $(6.139459, 4.733748)$ & $(4.246158, 8.010116)$ & 11.937125 & -0.044491800431819 & 0.1830 \\  
			4 & $(2.721648, 3.692260)$ & $(6.140468, 4.736302)$ & $(4.260762, 8.011354)$ & 11.935965 & -0.001159177048851 & 0.1182 \\  
			5 & $(2.722943, 3.690907)$ & $(6.140335, 4.735967)$ & $(4.264635, 8.011695)$ & 11.935946 & -0.000019629659851 & 0.1587 \\  
			6 & $(2.723145, 3.690696)$ & $(6.140406, 4.736145)$ & $(4.265266, 8.011751)$ & 11.935945 & -0.000000438958859 & 0.1825 \\  
			7 & $(2.723148, 3.690693)$ & $(6.140432, 4.736212)$ & $(4.265310, 8.011755)$ & 11.935945 & -0.000000014586059 & 0.1075 \\  
			8 & $(2.723147, 3.690695)$ & $(6.140439, 4.736228)$ & $(4.265330, 8.011756)$ & 11.935945 & -0.000000000348699 & 0.3371 \\  
			9 & $(2.723146, 3.690695)$ & $(6.140439, 4.736230)$ & $(4.265328, 8.011756)$ & 11.935945 & -0.000000000021878 & 0.0943 \\  
			10 & $(2.723146, 3.690695)$ & $(6.140439, 4.736230)$ & $(4.265328, 8.011756)$ & 11.935945 & -0.000000000000124 & 0.0902 \\  
			11 & $(2.723146, 3.690695)$ & $(6.140439, 4.736230)$ & $(4.265328, 8.011756)$ & 11.935945 & -0.000000000000004 & 0.3128 \\  
			12 & $(2.723146, 3.690695)$ & $(6.140439, 4.736230)$ & $(4.265328, 8.011756)$ & 11.935945 & 0.000000000000000 & 0.1061 \\ \hline
		\end{tabular}
		
	\end{sideways}
	\caption{Iteration-wise optimal points, distances, and convergence rate for the generalized waist problem with three discs, using Aitken acceleration}
	\label{tab:example1_aitken_data}
\end{table}

\newpage

\begin{example}
	We consider the generalized waist problem for three convex spheres in the Euclidean space \(\mathbb{R}^3\). Each sphere is defined by its center and radius:
	
	1. Sphere \(S_1\): Center \((2, 3, -1)\), Radius \(2\),
	\[
	S_1 = \{(x, y, z) \in \mathbb{R}^3 \mid (x - 2)^2 + (y - 3)^2 + (z + 1)^2 \leq 4\}.
	\]
	
	2. Sphere \(S_2\): Center \((4, -2, 1)\), Radius \(2\),
	\[
	S_2 = \{(x, y, z) \in \mathbb{R}^3 \mid (x - 4)^2 + (y + 2)^2 + (z - 1)^2 \leq 4\}.
	\]
	
	3. Sphere \(S_3\): Center \((6, 3, 2)\), Radius \(2\),
	\[
	S_3 = \{(x, y, z) \in \mathbb{R}^3 \mid (x - 6)^2 + (y - 3)^2 + (z - 2)^2 \leq 4\}.
	\]
	
	The objective is to find one point from each sphere, \(a_1 \in S_1\), \(a_2 \in S_2\), and \(a_3 \in S_3\), such that the total perimeter formed by connecting these points is minimized. The objective function is $D(a_1, a_2, a_3) = \|a_1 - a_2\| + \|a_2 - a_3\| + \|a_3 - a_1\|.$
\end{example}

To solve this optimization problem, we again employed the Projected Subgradient Descent (PSD) algorithm, as described in~\ref{alg:psd_algorithm}, implemented using MATLAB. The iterative process was initiated with starting points within each sphere: \(a_1^{(0)} = (3, 3, -1) \in S_1\), \(a_2^{(0)} = (5, -2, 1) \in S_2\), and \(a_3^{(0)} = (6, 4, 2) \in S_3\). Here, the step size parameter was set to a value of \(\alpha_k = 1.7432\) through an exact line search, as described in \ref{exactlinesearchtheorem}. The PSD algorithm effectively converged to an approximate optimal solution, yielding the points \(a_1^* = (2.8641, 2.9753, -0.6472)\), \(a_2^* = (4.7619, -1.9473, 1.4234)\), and \(a_3^* = (6.2548, 2.8465, 2.0186)\) on their respective spheres. As in the previous example, these points minimize the total perimeter, resulting in an optimal distance of \(D^* = 13.4567\). This outcome reinforces the effectiveness of the PSD algorithm in addressing the \textit{generalized waist problem}, even in three-dimensional space.

\ref{fig:optimal_points_spheres} provides a visual representation of the spheres \(S_1\), \(S_2\), and \(S_3\), the optimal points \((a_1^*, a_2^*, a_3^*)\), and the line segments connecting them. Correspondingly, \ref{tab:example2_iteration_data} provides a detailed analysis of the algorithm’s iterative performance, including the sequence of iterates, optimal points, the total distance, the change in total distance, the convergence rate , and the convergence rate \(\xi\). By presenting these results alongside the findings for the two dimensional case, we aim to highlight the versatility of the PSD algorithm across different geometric domains. Furthermore, to accelerate the convergence of the sequence of points, we implemented Aitken acceleration as discussed in~\ref{thm:aitken}. As shown in~\ref{tab:example2_Aitken_data}, Aitken's acceleration method reduces the number of iterations to 22 required for convergence, compared to 38 iterations required for convergence in the unaccelerated PSD algorithm.

\begin{figure}[H]
	\centering
	\includegraphics[width=0.9\textwidth]{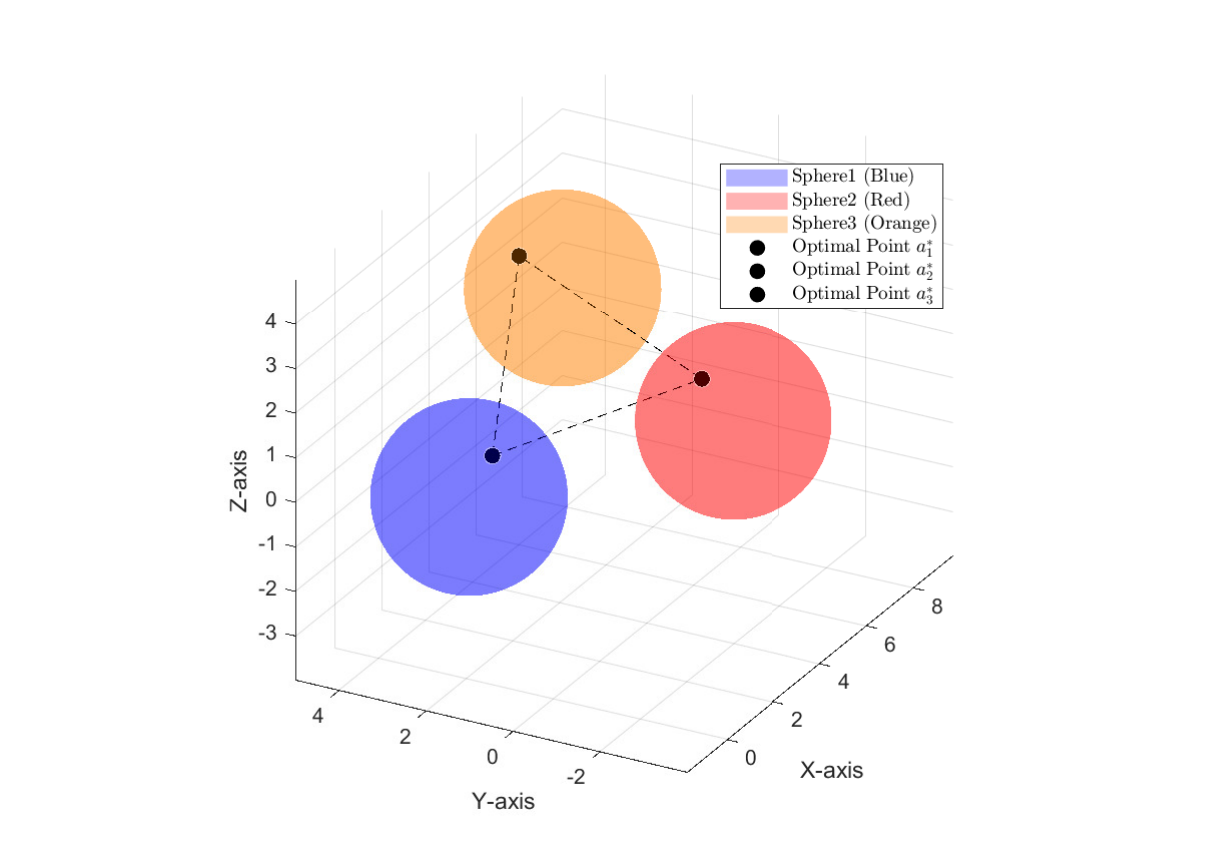}
	\caption{Optimal Points \(a_1^*\), \(a_2^*\), and \(a_3^*\) on Convex Spheres with Minimized Sum of Distances.}
	\label{fig:optimal_points_spheres}
\end{figure}

\begin{table}[H]
	\centering
	\small 
	\begin{sideways}
		\renewcommand{\arraystretch}{1.5} % Adjust row spacing
		\begin{tabular}{|c|c|c|c|c|c|c|}
			\hline
			$k$ & $(a_1)$ & $(a_2)$ & $(a_3)$ & $D_k$ & $\Delta D_k$ & $\xi=\varepsilon_{k+1}/\varepsilon_{k}$ \\ 
			\hline
			1 & $(3.78885, 3.00000, -0.10557)$ & $(4.81650, -0.36701, 1.81650)$ & $(6.00000, 4.78885, 2.89443)$ & 13.543440 & 13.543440379739803 & - \\ 
			2 & $(3.47830, 2.71033, 0.31557)$ & $(4.36223, -0.04351, 1.20230)$ & $(4.14272, 2.74603, 1.30286)$ & 7.015644 & 6.527796193268638 & 0.98067 \\ 
			3 & $(3.26667, 2.27893, 0.36953)$ & $(3.93336, -0.00225, 0.93248)$ & $(4.58982, 2.06012, 0.93791)$ & 6.063306 & 0.952338537549762 & 1.01597 \\ 
			4 & $(3.35594, 1.96113, 0.04028)$ & $(3.98949, -0.00692, 0.83410)$ & $(4.47105, 1.98434, 1.20580)$ & 5.909993 & 0.153312210017031 & 1.00823 \\ 
			5 & $(3.24817, 1.96343, 0.16944)$ & $(3.96211, -0.00864, 0.81823)$ & $(4.66319, 1.87515, 1.02653)$ & 5.872862 & 0.037131624226968 & 1.00092 \\ 
			6 & $(3.31831, 1.92285, 0.04966)$ & $(3.99666, -0.00974, 0.80285)$ & $(4.57938, 1.89120, 1.13259)$ & 5.860226 & 0.012636013077920 & 1.00029 \\ 
			7 & $(3.27864, 1.93568, 0.11008)$ & $(3.98527, -0.00982, 0.80263)$ & $(4.64410, 1.87190, 1.05716)$ & 5.855519 & 0.004706884093220 & 0.99998 \\ 
			8 & $(3.30730, 1.92655, 0.06709)$ & $(3.99601, -0.01034, 0.79693)$ & $(4.60858, 1.88036, 1.09981)$ & 5.853727 & 0.001792196061190 & 1.00000 \\ 
			
			15 & $(3.29799, 1.92960, 0.08142)$ & $(3.99401, -0.01037, 0.79672)$ & $(4.62465, 1.87681, 1.07975)$ & 5.852602 & 0.000002406174463 & 1.00000 \\ 
			
			20 & $(3.29845, 1.92952, 0.08080)$ & $(3.99412, -0.01038, 0.79661)$ & $(4.62409, 1.87690, 1.08048)$ & 5.852600 & 0.000000021932137 & 1.00000 \\ 
			
			25 & $(3.29840, 1.92953, 0.08086)$ & $(3.99411, -0.01038, 0.79662)$ & $(4.62414, 1.87689, 1.08041)$ & 5.852600 & 0.000000000200139 & 1.00000 \\ 
			
			30 & $(3.29841, 1.92953, 0.08085)$ & $(3.99411, -0.01038, 0.79662)$ & $(4.62414, 1.87690, 1.08041)$ & 5.852600 & 0.000000000001827 & 1.00000 \\ 
			
			35 & $(3.29841, 1.92953, 0.08085)$ & $(3.99411, -0.01038, 0.79662)$ & $(4.62414, 1.87690, 1.08041)$ & 5.852600 & 0.000000000000015 & 1.00000 \\ 
			36 & $(3.29841, 1.92953, 0.08085)$ & $(3.99411, -0.01038, 0.79662)$ & $(4.62414, 1.87690, 1.08041)$ & 5.852600 & 0.000000000000008 & 1.00000 \\ 
			37 & $(3.29841, 1.92953, 0.08085)$ & $(3.99411, -0.01038, 0.79662)$ & $(4.62414, 1.87690, 1.08041)$ & 5.852600 & 0.000000000000003 & 1.00000 \\ 
			38 & $(3.29841, 1.92953, 0.08085 )$ & $(3.99411, -0.01038, 0.79662 )$ & $(4.62414, 1.87690, 1.08041 )$ & 5.852600  & 0.000000000000000 & - \\ 
			\hline
		\end{tabular}
		
	\end{sideways}
	\caption{Iteration-wise optimal points and distances for the generalized waist problem with three sphere.}
	\label{tab:example2_iteration_data}
\end{table}

\begin{table}[H]
	\centering
	\small 
	\begin{sideways}
		\renewcommand{\arraystretch}{1.5} % Adjust row spacing
		\begin{tabular}{|c|c|c|c|c|c|c|}
			\hline
			$k$ & $(a_1)$ & $(a_2)$ & $(a_3)$ & $D_k$ & $\Delta D_k$ & $\xi=\varepsilon_{k+1}/\varepsilon_{k}$ \\ 
			\hline
			1  & $(3.46777, 2.47347, 0.25236)$ & $(4.25550, -0.01878, 1.09738)$ & $(4.36760, 2.23969, 1.12982)$ & 6.28685 & - & - \\ \hline
			
			2  & $(3.29829, 2.06772, 0.20220)$ & $(3.96781, -0.00475, 0.86613)$ & $(4.58245, 1.97471, 1.03081)$ & 5.89057 & 0.396280 & 0.2390 \\ \hline
			3  & $(3.31716, 1.95368, 0.08181)$ & $(3.98697, -0.00875, 0.81355)$ & $(4.58732, 1.91430, 1.09139)$ & 5.85913 & 0.031438 & 0.2645 \\ \hline
			4  & $(3.30023, 1.94364, 0.09248)$ & $(3.98872, -0.01072, 0.79349)$ & $(4.62505, 1.88217, 1.07265)$ & 5.85322 & 0.005908 & 0.3097 \\ \hline
			5  & $(3.30084, 1.92899, 0.07738)$ & $(3.99556, -0.01019, 0.79843)$ & $(4.62003, 1.88267, 1.07953)$ & 5.85270 & 0.000524 & 0.3978 \\ \hline
			
			10 & $(3.29847, 1.92953, 0.08078)$ & $(3.99413, -0.01038, 0.79657)$ & $(4.62416, 1.87687, 1.08041)$ & 5.85260 & 0.000000 & 0.5354 \\ \hline
			
			15 & $(3.29841, 1.92953, 0.08085)$ & $(3.99411, -0.01038, 0.79662)$ & $(4.62414, 1.87690, 1.08041)$ & 5.85260 & 0.000000 & 0.4313 \\ \hline
			16 & $(3.29841, 1.92953, 0.08085)$ & $(3.99411, -0.01038, 0.79662)$ & $(4.62414, 1.87690, 1.08041)$ & 5.85260 & 0.000000 & 0.4109 \\ \hline
			17 & $(3.29841, 1.92953, 0.08085)$ & $(3.99411, -0.01038, 0.79662)$ & $(4.62414, 1.87690, 1.08041)$ & 5.85260 & 0.000000 & 0.5222 \\ \hline
			18 & $(3.29841, 1.92953, 0.08085)$ & $(3.99411, -0.01038, 0.79662)$ & $(4.62414, 1.87690, 1.08041)$ & 5.85260 & 0.000000 & 0.3991 \\ \hline
			19 & $(3.29841, 1.92953, 0.08085)$ & $(3.99411, -0.01038, 0.79662)$ & $(4.62414, 1.87690, 1.08041)$ & 5.85260 & 0.000000 & 0.6062 \\ \hline
			20 & $(3.29841, 1.92953, 0.08085)$ & $(3.99411, -0.01038, 0.79662)$ & $(4.62414, 1.87690, 1.08041)$ & 5.85260 & 0.000000 & 0.1956 \\ \hline
			21 & $(3.29841, 1.92953, 0.08085)$ & $(3.99411, -0.01038, 0.79662)$ & $(4.62414, 1.87690, 1.08041)$ & 5.85260 & 0.000000 & 0.7820 \\ \hline
			22 & $(3.29841, 1.92953, 0.08085)$ & $(3.99411, -0.01038, 0.79662)$ & $(4.62414, 1.87690, 1.08041)$ & 5.85260 & 0.000000 & 0.4510 \\ \hline
		\end{tabular}
		
	\end{sideways}
	\caption{Iteration-wise pptimal points, distances, and convergence rate for the generalized waist problem with three spehere, using Aitken acceleration}
	\label{tab:example2_Aitken_data}
\end{table}

\newpage

In~\ref{tab:psd-nag-comparison-disc}, we present detailed numerical results for~\ref{disc_example}, which involves three convex disc constraints. These results illustrate the deep impact of different step size choices on the performance of our projected subgradient descent (PSD) method. Specifically, the optimal step size \( \alpha = 2.07 \) was selected using the exact line search strategy detailed in~\ref{exactlinesearchtheorem}. The data clearly shows significant efficiency improvements: employing Aitken's acceleration  drastically reduces the iteration count from hundreds of thousands to merely 12 iterations and CPU time from over 1400 seconds to fractions of a second. This highlights the critical importance of carefully choosing step sizes and utilizing acceleration techniques to substantially decrease computational costs.

For comparison, Mordukhovich et al.~\cite{MoNaSa12} utilized a similar projected subgradient approach but with a traditional diminishing step size \( \alpha = \frac{1}{k} \) for solving the Heron problem constrained by four squares, requiring approximately 600,000 iterations for convergence. In stark contrast, our method, using the exact line search approach, achieves convergence in only 19 iterations, demonstrating the significant advantages of advanced step size determination techniques.
\begin{table}[htbp]
	\centering
	
	\begin{tabular}{l c r r}
		\toprule
		\textbf{Method} & \textbf{Tolerance} & \textbf{\# Iterations} & \textbf{CPU Time (s)} \\
		\midrule
		PSD ($\alpha = \frac{1}{k}$)                & $1\times 10^{-15}$ & 343406 & 1449.26313 \\
		PSD ($\alpha = \frac{1}{k}$)                & $1\times 10^{-14}$ & 176674 &   11.21159 \\
		PSD ($\alpha =\frac{1}{k}$)                & $1\times 10^{-13}$ &  88090 &    6.36560 \\
		PSD ($\alpha = \frac{1}{k}$)                & $1\times 10^{-12}$ &  38956 &    2.40782 \\
		PSD ($\alpha = \frac{1}{k}$)                & $1\times 10^{-11}$ &  17077 &    1.12642 \\
		PSD ($\alpha = \frac{1}{k}$)                & $1\times 10^{-10}$ &   7557 &    0.50496 \\
		PSD ($\alpha = 0.01$)             & $1\times 10^{-15}$ &   1877 &    0.11776 \\
		PSD ($\alpha = 0.1$)              & $1\times 10^{-15}$ &    205 &    0.02991 \\
		PSD ($\alpha = 2.07$)             & $1\times 10^{-15}$ &     19 &    0.01772 \\
		PSD with Aitken's Acceleration    & $1\times 10^{-15}$ &     12 &    0.00013 \\
		\midrule
		NAG ($\alpha = 0.1, t = 1$)       & $1\times 10^{-15}$ &    217 &    0.09325 \\
		NAG ($\alpha = 0.01, t = 1$)      & $1\times 10^{-15}$ &   1457 &    0.41198 \\
		NAG ($\alpha = 2.07, t = 1$)      & $1\times 10^{-15}$ &     21 &    0.04909 \\
		\bottomrule
	\end{tabular}
	\caption{Performance comparison of Projected Subgradient Descent (PSD) and Nesterov’s Accelerated Gradient (NAG) methods with varying step sizes for the generalized waist problem presented in~\ref{disc_example}}
	\label{tab:psd-nag-comparison-disc}
\end{table}

\section{Conclusion}\label{conclusion}	

In this paper, we have introduced a generalized version of the classical waist problem by replacing the lines with convex sets, thereby broadening its scope and applicability. We established both the existence and uniqueness of solutions under mild geometric conditions. We derived necessary and sufficient optimality conditions with a compelling geometric interpretation linked to physical laws such as reflection. Furthermore, we presented a robust computational framework based on projected subgradient descent method, supported by illustrative numerical examples that validate the efficiency and accuracy of our approach. The island analogy offers an intuitive understanding of the problem's practical implications, especially in network design and facility location. In the generalized waist problem, if there are only three convex sets as line segments making a triangle, the generalized waist problem becomes to the classical Fagnano problem. Thus, as a future scope, it may be challenging to investigate the generalized waist problem when they are not pairwise disjoint. Future work may also explore extensions to Hilbert spaces or non-convex constraints, offering promising directions for both theoretical exploration and practical application.

\bibliographystyle{plain}
\bibliography{Waist.bib}		
\end{document}